\RequirePackage{fix-cm}

\documentclass[smallextended]{svjour3a}
\smartqed  
\usepackage{graphicx}
\usepackage{mathptmx}      
\usepackage{amsfonts}
\usepackage{amstext}
\usepackage{amsmath}
\usepackage{amscd}
\usepackage{amssymb}
\usepackage[mathscr]{eucal}
\usepackage{url}
\usepackage{epsf}

\numberwithin{equation}{section}

\newcommand\mL{L\kern-0.08cm\char39}

\newcommand\id{\mathop{\rm id}}
\newcommand\card{\mathop{\rm card}}

\newcommand\Inte{\mathop{\rm Int}}

\newcommand\Lint{\mathop{\rm Sint}}
\newcommand\End{\mathop{\rm End}}

\newcommand\pr{\mathop{\rm pr}}
\newcommand\ord{\mathop{\rm ord}}

\newcommand\SL{\mathop{\rm SL}}
\newcommand\rot{\mathop{\rm rot}}

\begin{document}

\title{Minimal sets of fibre-preserving maps in graph bundles
}

\author{S. Kolyada         \and
       {\mL}. Snoha \and S. Trofimchuk
}

\institute{S. Kolyada \at
              Institute of Mathematics, NASU, Tereshchenkivs'ka 3,
01601 Kiev, Ukraine \\
              \email{skolyada@imath.kiev.ua}           
           \and
           {\mL}. Snoha  \at
              Department of Mathematics, Faculty of Natural Sciences,
            Matej Bel University, Tajovsk\'eho 40, 974 01 Bansk\'a Bystrica,
            Slovakia\\
              \email{Lubomir.Snoha@umb.sk}           
           \and
           {S}. Trofimchuk \at
              Instituto de Matem\'atica y F\'{\i}sica, Universidad
de Talca, Casilla 747,
Talca, Chile \\
              \email{trofimch@inst-mat.utalca.cl}           
}

\date{}

\maketitle

\begin{abstract}
Topological structure of minimal sets is studied for a dynamical system $(E,F)$ given by a fibre-preserving, in general non-invertible,  continuous selfmap $F$ of a graph bundle $E$. These systems include, as a very particular case, quasiperiodically forced circle homeomorphisms. Let $M$ be a minimal set of $F$ with full projection onto the base space $B$ of the bundle. We show that $M$ is nowhere dense or has nonempty interior depending on whether the set of so called endpoints of $M$ is dense in $M$ or is empty.
If $M$ is nowhere dense, we prove that either a typical fibre of $M$ is a Cantor set, or there is a positive integer $N$ such that a typical fibre of $M$ has cardinality $N$.
If $M$ has nonempty interior we prove that there is a positive integer $m$ such that a typical fibre of $M$, in fact even each fibre of $M$ over a \emph{dense open} set $\mathcal O \subseteq B$, is a disjoint union of $m$~circles. Moreover, we show that each of the fibres of $M$ over $B\setminus \mathcal O$ is a union of circles properly containing a disjoint union of $m$~circles. Surprisingly,  some of the circles in such ``non-typical" fibres of $M$ may intersect. We also give sufficient conditions for $M$ to be  a sub-bundle of~$E$.\keywords{Dynamical system \and minimal set \and  graph bundle \and skew product}
\noindent {\bf Mathematics Subject ClassiÞcation (2010)}  
{Primary 54H20; Secondary 37B05}
\end{abstract}

\section{Introduction and statement of main results}\label{S:intro}

Let $f$ be a continuous selfmap of a compact metric space $X$. The~system $(X,f)$ is called \emph{minimal} if there is
no proper subset $M\subseteq X$ which is nonempty, closed and
$f$-invariant (i.e., $f(M)\subseteq M$). In such a case we also say
that the map $f$ itself is minimal. Clearly, a system $(X,f)$ is
minimal if and only if the orbit  $\{x, f(x), f^2 (x), \dots \}$ of every point $x\in X$ is dense in
$X$.  Note that by an orbit
we mean a forward orbit rather than a full orbit, even if $f$ is a
homeomorphism.  The basic  fact  is that any compact dynamical
system $(X,f)$ has minimal (closed) subsystems $(M, f|_{M})$. Such closed
sets $M$ are called \emph{minimal sets} of $f$ or, more precisely,
of $(X,f)$. 
The minimal sets, as `irreducible' parts of a system, attract much
attention and their topological structure is one of the central topics
in topological dynamics.

The classification  of compact metric spaces admitting minimal maps is a
well-known open problem in topological dynamics \cite{Aus,Bron}.
For the state of the art of the problem see~\cite{BDHSS,B,BKS,JKP} and references therein. 

It is folklore that if
$X$ is a compact \emph{zero-dimensional} space, $f: X\to X$ is continuous and
$M\subseteq X$ is a minimal set of $f$ then $M$ is either a finite set (a
periodic orbit of $f$) or a Cantor set and this is in fact a characterization
because also conversely, whenever $M\subseteq X$ is a finite or a Cantor set then
there is a continuous map $f: X\to X$ such that $M$ is a minimal set of $f$.
Among \emph{one-dimensional} spaces, the characterization of minimal sets is
known for \emph{graphs} --- minimal sets on graphs are  finite sets, Cantor sets and unions of finitely many pairwise disjoint
simple closed curves, see \cite{BHS} or \cite{Mai}.  The full
characterization of minimal sets on (local)  \emph{dendrites} has been found just recently 
\cite{BDHSS}.

In \emph{higher dimensions} the topological structure of minimal
sets is much more complicated and, besides some important examples (see e.g. \cite{BKS,Hand}),
only few results are known.  One obvious fact is that  if $h$ is a homeomorphism 
of a connected space $X$ then a minimal set of $h$ either is nowhere dense or 
coincides with  $X$.   It is interesting that  the same conclusion  is true  for continuous 
endomorphisms of compact connected 2-manifolds \cite{KST-mnfld}
while  it is an open problem whether this result  holds also in 
dimensions $n >2$.  A related question is which  manifolds admit minimal maps. 
Again, the answer is completely known only in dimension 2: among 2-manifolds, compact or not, 
with or without boundary, only finite unions of tori and finite unions  of Klein bottles admit minimal 
maps \cite{B}.   In dimensions
higher than $2$ the tori and we know from \cite{FH} that also the
odd-dimensional spheres admit minimal diffeomorphisms.  
 Note that a non-compact manifold never admits a minimal map 
by \cite{Got}. This is  because we define minimality as density of  {\it forward} orbits.  It does not exclude the possibility  to have a homeomorphism of a non-compact manifold with all {\it full} orbits dense.  In any case, 2-sphere without  a finite set of points does not admit such a homeomorphism  \cite{LCY}.  

To find a full
topological characterization of minimal sets on compact, connected
2-manifolds is a very difficult task.  Very recently, a classification of minimal sets on 2-torus has been obtained  for homeomorphisms \cite{JKP}.  

The main contribution of the present paper is a partial description of minimal sets of  fibre-preserving maps in graph bundles.

\subsection{Fibre-preserving maps and their minimal sets}\label{SS:motivation}
A dynamical system $(E,F)$ is called an
\emph{extension} of a \emph{base} dynamical system $(B,f)$ if there is a
continuous surjective map $p: E\to B$, called a {\it factor map} or a {\it projection},  such that $p \circ F = f \circ p$.
We also say that the base $(B,f)$ is a \emph{factor} of $(E,F)$.   Note that for every $b\in B$ we have
$F(p^{-1}(b)) \subseteq p^{-1}(f(b))$, i.e., $F$ sends the fibre over
$b$ into the fibre over $f(b)$. Therefore $F$ is said to be \emph{fibre-preserving}.
 Suppose that $(B,f)$ is minimal and $(E,F)$ is an extension of it. If we additionally assume that $E$ is compact then always there
is a minimal set $M$ in the system $(E,F)$ and since $M$ projects onto a minimal set of $(B,f)$,
we necessarily have $p(M)=B$. 

A very special case of an extension is when $E$ is a cartesian product, $E= B\times Y$, and $F(x,y) =
(f(x),g(x,y))$. Then $F$ is fibre-preserving, the fibres being the ``vertical" copies of $Y$, i.e.
the sets $\{b\}\times Y$ where $b\in B$, and the factor map being the natural projection of $E$
onto $B$. The map $F$ is also called a \emph{skew product map} or
sometimes a \emph{triangular map}.

The study of fibre preserving maps and  their mi\-ni\-mal sets
has a long tradition. 
Much attention has been paid to minimal sets of fibre-preserving maps on the torus, for instance in the case of \emph{quasiperiodically forced (qpf) circle homeomorphisms}. These systems naturally appear in the  study of the scalar linear quasi-periodic Schr\"o\-din\-ger equations. In such a case the dynamics is given by the \emph{projective action of a quasiperiodic $\SL(2,\mathbb R)$-cocycle} (the $2$-torus is identified with $\mathbb S^1 \times \mathbb P^1(\mathbb R)$ and the projective action of $\SL(2,\mathbb R)$ is considered on $\mathbb P^1(\mathbb R)$).  The most interesting situation occurs when the mentioned Schr\"odinger equations are non-uniformly hyperbolic \cite{Bjer}. An old question by Herman \cite[Section 4.14]{H} concerns   topological structure of  the unique minimal set $M$ in that case.  Herman partially described the set $M$. In particular, $M$ is nowhere dense and the intersection $M_\theta$ of $M$ with a vertical fibre $\{\theta\}\times  \mathbb P^1(\mathbb R)$ is, generically, a singleton. Herman's question is  whether also all the other fibres  $M_\theta$ are connected;  for more details  and related results see~ \cite{BCJL}, \cite{BJohn},  \cite{Bjer}, \cite{H} and references therein.  Bjerkl\"ov \cite{Bjer} shows that the question has an affirmative answer in some special cases.
According to recent preprint  by Hric and J\"ager \cite{HrJ},  in general the answer is negative.

 In a more general setting of \emph{skew product circle flows} (both continuous and discrete) over  a minimal base (forcing) on a compact metric space $Y$, a topological classification of minimal sets was recently  given by Huang and Yi   \cite{HY}.  They showed that if $M$ is a minimal set of such a system then either $M$ is the whole space $Y\times \mathbb S^1$, or there is a positive integer $N$ such that a typical  fibre of $M$ consists of $N$ points, or a typical fibre of $M$ is a Cantor set. Below in Theorem~E, we amplify this result  to general fibre-preserving (not necessarily invertible) maps in compact graph bundles over  a minimal base.

B\'eguin, Crovisier,  J\"ager and Le Roux~\cite{BCJL} have constructed   transitive qpf circle homeomorphisms with complicated minimal sets. For example, the minimal set can be a Cantor set whose intersection with each vertical fibre (circle) is uncountable (the possibility that some of these intersections have isolated points in the topology of the fibre has not been excluded and is  probable). Thus, minimal sets of fibre preserving maps can be quite complicated. This is true even for triangular maps in the square.
To illustrate this, recall that so called \emph{Floyd-Auslander minimal systems} \cite{HJ} are
homeomorphisms which are extensions of Cantor minimal homeomorphisms and their
phase spaces are subsets of the unit square which are nonhomogeneous --- some
fibres are compact intervals while the others are singletons. Modifying the
construction, one can obtain also a noninvertible nonhomogeneous system of this
kind \cite{SS}. Note that, by the extension lemma from \cite{KS1}, all these
systems can be embedded into systems given by triangular selfmaps of the square.

In  the present paper we wish to shed more light on the problem of characterizing  minimal
sets of higher dimensional maps by studying minimal sets of continuous fibre-preserving
(not necessarily invertible) maps in graph bundles. It does not seem easy to generalize the results to more general bundles.

\vspace{0mm}

\subsection{Star-like interior points and end-points in graph bundles}\label{SS:terminology}

To state our main results, we need  some terminology. A \emph{fibre space} is an object
$(E, B, p)$ where $E$ and $B$ are topological spaces and $p: E \to
B$ is a continuous surjection. Here $E$, $B$ and $p$ are called the
total space, the base (space) and the projection (map) of the fibre
space, respectively, and $p^{-1}(b)$ is called the fibre over the
point $b\in B$. If $\Gamma$ is another topological space, the fibre
space $(E, B, p)$ is called a \emph{fibre bundle} with fibre $\Gamma$, and denoted by $(E, B, p,
\Gamma)$, if the projection map $p: E \to B$ satisfies the following
condition of local triviality: For every point $b \in B$ there is an
open neighborhood $U$ of $b$ (which will be called a
\emph{trivializing neighborhood}) and a homeomorphism $h:
p^{-1}(U)\to U\times \Gamma$ such that on $p^{-1}(U)$ it holds
$\pr_1 \circ h = p$. Here $\pr_1: U\times \Gamma\to U$ is the
canonical projection onto the first factor.  We will always assume that {\bf both $E$ and $B$ are compact metric
spaces}  and so we will speak on compact
fibre bundles.

Given a fibre space  $(E, B, p)$, consider  dynamical systems
$(E,F)$ and $(B,f)$ with $p\circ F = f \circ p$. Thus, $(E,F)$
is an extension of $(B,f)$ and $(B,f)$ is a factor of $(E,F)$, the
projection map $p$ being the factor map. Then  $F$ is \emph{fibre-preserving}, it sends
the fibre $p^{-1}(b)$ over $b\in B$ into the fibre $p^{-1}(f(b))$ over $f(b)$.

A \emph{graph} is a (nonempty) compact metric space which can be
written as the union of finitely many arcs any two of which are either disjoint
or intersect only in one or both of their end-points. A graph need
not be connected and  a singleton is not a graph.   A
\emph{tree} is a graph containing no {\it circle} (i.e. a {simple closed curve}). The number
of arcs emanating  from a point $x \in G$ is called the
\emph{order} of $x$ and is denoted by $\ord (x, G)$. Points of
order $1$ are called \emph{end-points} of $G$ and points of order at
least $3$ are called \emph{ramification points} of $G$.

For $n\geq 1$ we will consider the notion of the \emph{$n$-star}
$S_n$, which can be described as the set of all complex numbers $z$
such that $z^n$ is in the real unit interval $[0,1]$, i.e., a
central point (the origin) with $n$ copies of the interval $[0,1]$
attached to it. We will view the $n$-star as a tree with $n+1$
vertices, one of them (the central point) having order $n$ and the
other $n$ vertices (the end-points of $S_n$) having order $1$. Any set homeomorphic to $S_n$ will also be called an
$n$-star and also denoted by $S_n$. Note that both  $S_1$ and $S_2$ are
homeomorphic to a closed interval. By the
\emph{open $n$-star} $\Sigma_n$ we will mean $S_n$ without its $n$
end-points. In particular, $\Sigma_2$ is homeomorphic to an open
interval (while $\Sigma_1$ to a half-closed interval).

\begin{definition}\label{D:star-like in graph}
Let $\Gamma$ be a graph and $Z\subseteq \Gamma$ be closed. A point
$x\in Z$ is said to be a \emph{star-like interior point of $Z$} if
there exists a $Z$-open neighborhood of $x$ (i.e., the intersection
of $Z$ and a $\Gamma$-open neighborhood of $x$) which is
homeomorphic to $\Sigma_k$ for some $k\geq 2$; we assume here that this
homeomorphism sends the point $x$ to the central point of $\Sigma_k$
(then $k$ is uniquely determined).
If $x\in Z$ is not a star-like interior point of $Z$ we say that it is an
\emph{end-point of $Z$}. Let $\Lint (Z)$ and $\End (Z)$ denote the
set of all star-like interior points of $Z$ and the set of all end
points of $Z$, respectively.
\end{definition}

Figure 1 shows that a star-like interior point of $Z$ need not be an
interior point of $Z$ in  $\Gamma$ and an interior point
of $Z$ need not be a star-like interior point of $Z$.

\vspace{-5mm}

 \begin{figure}[ht]
  \label{Fig:1}
  \begin{center}
  \includegraphics[width=8.0cm]{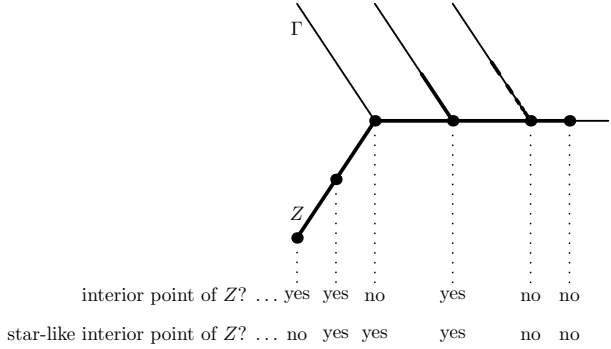}
   \caption{There is no connection between interior and star-like interior points.}
  \end{center}
 \end{figure}

\vspace{-5mm}

The set $\Lint (Z)$ is open in $Z$ (but
not necessarily in $\Gamma$) and so the set $\End (Z)$ is closed in
$Z$ (hence closed in $\Gamma$). If $Z$ is a subgraph of $\Gamma$,
the set $\End (Z)$ coincides with the usual set of end-points of the
graph $Z$.

A \emph{graph bundle} is a fibre bundle whose fibre $\Gamma$ is a
graph. Given a graph bundle $(E, B, p, \Gamma)$, for $M\subseteq E$
and $b\in B$ we denote $M_b = M \cap p^{-1}(b)$; this set is said to
be the \emph{fibre of $M$ over $b$}. When speaking on the fibres of $M$
over points lying in a subset $U$ of $B$,
we sometimes call them \emph{fibres of $M$ over the set $U$}.
If $M\subseteq E$ and $U\subseteq B$, we denote $M_U = M \cap p^{-1}(U)$.
So, $M_U$ is the union of all fibres of $M$ over the set $U$.

\begin{definition}\label{D:star-like in bundle}
Given a closed set $M$ in a compact graph bundle $(E, B, p, \Gamma)$ we define the
set of \emph{star-like interior points of $M$} and the set of \emph{end-points of $M$} by, respectively, 
$$
\Lint (M) = \bigcup_{b\in B} \Lint (M_b) \quad \text{and} \quad \End (M) = \bigcup_{b\in B} \End (M_b).
$$
\end{definition}

Of course, $\End (M) = M\setminus \Lint (M)$. In general it is not true
that $\Lint (M)$ is open or $\End (M)$ is closed in $E$ or $M$.

\subsection{Main results}\label{SS:main results}
Throughout the paper,  $(E, B, p, \Gamma)$ is a compact graph bundle, $(E,F)$ and $(B,f)$ are 
dynamical systems with $p\circ F = f \circ p$. We also assume that  the base system $(B,f)$ is minimal or,  equivalently, that $p(M)=B$ for each minimal set $M \subseteq E$ of  $F$.  Our first main result is the following dichotomy for a minimal set $M$ formulated in terms of end-points of $M$.

\vspace{2mm}

\noindent {\bf Theorem A.} {\it
Let $M$ be a minimal set (with full projection onto the base) of a fibre-preserving map in a compact graph bundle $(E,B,p,\Gamma)$. Then there are two mutually exclusive possibilities:
\begin{enumerate}
\item [{ (A1)}] $\overline {\End (M)}=M$ (and this holds if and only if $M$ is nowhere dense in $E$);
\item [{ (A2)}] $\End(M) = \emptyset$ (and this holds if and only if $M$ has nonempty interior in $E$). 
\end{enumerate}
In particular, the fibre-preserving maps in tree bundles have only nowhere dense minimal sets.}

\vspace{2mm}

The assumption that the base system $(B,f)$ is minimal is not
restrictive. In fact, if $M$ is a minimal set of $(E,F)$ then its
projection $p(M)$ is a minimal set of $(B,f)$ and so one
can pass to the sub-bundle over $p(M)$ and to consider, instead of
$(E,F)$, the system $(E^*, F|_{E^*})$ where $E^* = p^{-1} (p(M))$.
As an application of this fact we get that though a minimal
set of a triangular  map in the square can
contain a vertical interval (so that in general $\End (M) \neq M$
in the case (A1)), the following corollary
holds ($I$~denotes a real compact interval and $\pr_1$ is the projection
onto the first coordinate).

\vspace{2mm}

\noindent {\bf Corollary B.} {\it
Let $F(x,y)=(f(x), g(x,y))$ be a continuous triangular map in
the square $I^2$ and let $M$ be a minimal set of $F$. Then $M$ is
nowhere dense in the space $\pr_1(M) \times I$.
}
\vspace{2mm}

We know from the characterization of minimal sets on the interval
that $\pr_1(M)$ is either a finite set or a Cantor set. In the
latter case the result in the corollary is nontrivial, it
strengthens Theorem~1 from~\cite{FPS2} (where the same result is
obtained for a very particular and small subclass of the class of
triangular selfmaps of the square) and answers in negative the
question posed by J.~Sm\'\i tal whether a minimal set $M$ of a
triangular map in the square can have nonempty interior in the space
$\pr_1(M) \times I$.

\medskip

So,  no direct-product $B\times I$  admits a minimal fibre-preserving map (with the fibre $I$). Cannot we remove the assumption that the maps are fibre-preserving? The answer is negative. In fact, if $\mathbb S^1$ is a circle and $H$ is the Hilbert cube then the space $P= \mathbb S^1 \times H$ admits a continuous minimal map (in the form of a skew product map with an irrational rotation in the base $\mathbb S^1$ and homeomorphisms $H\to H$ as fibre maps, see \cite{GW}). However, $P$ can be written in the form $P= (\mathbb S^1 \times H) \times I$. 
Thus we have a space of the form $B\times I$ admitting a minimal, of course not fibre-preserving map (with the fibre being $I$).  Here dimension of $B$ is infinite. An interesting 
 question is whether it is true that all minimal, not necessarily fibre-preserving, maps in interval bundles $B\times I$ have only nowhere dense minimal sets if we additionally assume that $B$ has finite dimension. Recall that, by the result from~\cite{KST-mnfld}, this is true if $B$ is a one-dimensional manifold, possible with boundary, so that $B\times I$ is a $2$-manifold with boundary.

\medskip

In each of the cases (A1) and (A2) in Theorem~A, there are severe restrictions for the topological structure of the minimal set $M$.
In the case (A2), some of such restrictions   are listed in Theorem~C whose \emph{full version} is given in Section~\ref{S:proofC}. Here, in Introduction, we prefer to list just those of them which seem to be most important and whose statement is neither cumbersome nor involves the notion of \emph{strongly} star-like interior points which will be introduced in Section~\ref{S:slip}. To keep the shortened version of the theorem compatible with the full version, we do not renumber the items.

\vspace{2mm}

\noindent {\bf Theorem C (shortened version).} {\it Let $M$ be a minimal set (with full projection onto the base) of a fibre-preserving map in a compact graph bundle $(E,B,p,\Gamma)$. Assume that $M$ has nonempty interior.  Then the following holds.
\begin{enumerate}
\item [(C4)] All the sets $M_b$, $b\in B$, are unions of circles. In fact there exist an open dense set  $\mathcal{O} \subseteq B$ and a positive integer $m$ such that
\begin{itemize}
\item[$\bullet$] for each $z \in \mathcal{O}$, $M_z$ \emph{is} a \emph{disjoint} union of $m$ circles, and
\item[$\bullet$] for each $z\in B\setminus \mathcal{O}$, $M_z$ is a union of circles which properly contains a disjoint union of $m$ circles.
\end{itemize}
In particular, if $M_z$ is a circle for some $z\in B$, then $M_z$ is a circle for all $z$ in the open dense subset $\mathcal O$ of $B$.
\item [(C6)] The set $M_{\mathcal O}$ is dense in $M$.
\item [(C8)] If $z\in \mathcal{O}$ then the set $M_z$, which is a disjoint union of $m$ circles, is mapped by $F$ onto a disjoint union of $m$ circles in $M_{f(z)}$.
\item [(C10)] If $f$ is monotone  then  $\mathcal{O} = B$ (hence, $M$ is a sub-bundle of $E$).
\item [(C11)] If $E = B\times \Gamma$ and $B$ is locally connected then $\mathcal{O} = B$ (hence, $M$ is a sub-bundle of~$E$ and if $B$ is also connected, then $M$ is a direct product).
\end{enumerate} }

\vspace{2mm}

Properties of the map $F|_M$  are  partially described in  
 Proposition~\ref{P:monot}.   The next result shows that  $\mathcal O \neq B$ is possible and that some circles in a fibre of $M$ over a point in $B\setminus \mathcal O$ can intersect. 
 
 \vspace{2mm}

\noindent {\bf Theorem D.} {\it
 There is a minimal selfmap $f$ of a Cantor set $B$, a connected  graph $\Gamma$ and an extension $(B\times \Gamma, F)$ of $(B,f)$ with  a minimal set $M$ such that, for some $b\in B$, 
\begin{itemize}
\item[$\bullet$] $M_z$ is a circle for each $z \not=b$, and  
\item[$\bullet$] $M_b$ is a union of two circles. 
Depending on the choice of such a system,  the union of any two different circles in any graph 
can appear as the set $M_b$. 
\end{itemize}}

\vspace{2mm}
 
Recall that a set in a Baire space is called \emph{residual} if its complement is of {\it 1st category}, i.e.   a countable union of nowhere dense sets. 
By saying that a \emph{typical} (or {\it generic}) fibre of $M$ has some property we mean that there is a residual set in the base $B$ such that for each $b$ in this residual set, the fibre $M_b$ of $M$ has this property. 

Notice that Theorem~C, part (C4),  describes  a typical fibre of the minimal set $M$ in the case (A2).  Also in the case (A1) we are able to describe a typical fibre of $M$. 

\vspace{2mm}

\noindent {\bf Theorem E.} {\it
Let $M$ be a minimal set (with full projection onto the base) of a fibre-preserving map in a compact graph bundle $(E,B,p,\Gamma)$. Assume that $M$ is nowhere dense.  Then either
\begin{enumerate}
\item [(E1)] a typical fibre of $M$ is a Cantor set, or
\item [(E2)] there is a positive integer $N$ such that a typical fibre of $M$ has cardinality $N$.
\end{enumerate}}
\vspace{2mm}

The number $N$ in (E2) is given by the formula from Proposition~\ref{P:D} in Section \ref{S:proofsDE}. Even if $F$ is a homeomorphism, one cannot claim that all fibres of $M$ have the same cardina\-lity, see examples in the next section.

In the special case when $E$ is a direct product $B\times \Gamma$,   $\Gamma$ is the circle and $F: E\to E$ is a homeomorphism, Theorem E has been known from~\cite[Theorem 6.1]{HY}.

Notice the following asymmetry: in the case (A2) we know from (C4) that a ``non-typical" fibre of $M$ is a union of circles,   while in the  case (A1)  the topological structure of a ``non-typical" fibre is unknown 
even for  the qpf circle homeomorphisms and the triangular maps in the square  
(as Floyd-Auslander systems show, some of these fibres can contain nondegenerate intervals).

The paper is organized as follows. In Section~\ref{S:remarks} we present several 
illustrating examp\-les of  minimal sets of fibre-preserving maps in graph bundles and 
we also prove Theorem D.  Section~\ref{S:prelim} contains some dynamical and topological
preliminaries. Then, in Section~\ref{S:slip} we introduce the key notion
of our paper, namely that of a strongly star-like interior point of
a subset of a graph bundle, and we study the structure of open
neighborhoods of those compact subsets of a fibre which entirely
consist of strongly star-like interior points of a given subset of
the bundle.   The proofs of Theorems~A, C and E  are given in
Sections~\ref{S:proofA}, \ref{S:proofC} and \ref{S:proofsDE}, respectively.

\section{Some examples and proof of Theorem D}\label{S:remarks}

Theorems~C and~E give necessary conditions for subsets of graph bundles to be minimal for a fibre-preserving map. Observe the following.

Suppose that the base $B$ is a singleton and so $E$ is just $\Gamma$. Then  Theorems~C and~E  imply that minimal sets on graphs are finite sets, Cantor sets or disjoint unions of (finitely many) circles. This is already a characterization of minimal sets on graphs, as shown in~\cite{BHS} or~\cite{Mai}. If $B$ is finite (and so the minimal base system  is just a periodic orbit) we get that each fibre of $M$ either is a Cantor set or consists of the same finite number of points or the same finite number of disjoint circles. Again, one can easily show that this is a characterization of minimal sets (with full projection) of fibre-preserving maps in graph bundles with finite base.

However, we do not know how far we are from a  topological characterization of minimal sets (with full projection) of fibre-preserving maps in graph bundles with infinite  base. Indeed,  if typical fibres of  some compact  set $M\subseteq E$ are as described in Theorems~C and~E  (and $M$ has no isolated point,  which  would be an obvious obstacle for $M$ to be minimal)  then it is not easy to
check when there exists a fibre-preserving map $F$ in $E$ such that $M$ is a minimal set of $F$.

\subsection{Examples  of nowhere dense minimal sets}
Only nowhere dense minimal sets can appear if $\Gamma$ is a tree. Say, a triangular map in the square
can have a minimal set which is the direct product of a Cantor set with itself. More interesting
are the following examples of nowhere dense minimal sets  which are not totally disconnected.

\begin{example}[\emph{Floyd-Auslander minimal sets}].\label{E:FA-like}
By the extension lemma from~\cite{KS1} one can extend any Floyd-Auslander minimal system $(M,H)$ (see~\cite{HJ}) to a triangular map defined on the product of the Cantor set (the projection of $M$) and a compact interval. Though in this example $H$ is a homeomorphism on $M$, it is not true in general that if $f$ is a homeomorphism then $F|_M$ is monotone --- to see it, replace $(M,H)$ in this construction by a  noninvertible modification of it from~\cite{SS}. Other examples can be obtained in a similar way, by replacing a Floyd-Auslander minimal system by some other cantoroids (for the definition of a cantoroid see~\cite{BDHSS}). \hfill $\square$
\end{example}

\begin{example}[\emph{Boundary of the M\"obius band as a minimal set}].\label{E:Mobius}
Imagine, in $\mathbb R^3$, a circle $\mathbb S^1$ in a horizontal plane and a vertical straight line segment $I$ whose center is a point of $\mathbb S^1$ and the length of $I$ is smaller than the radius of $\mathbb S^1$. By moving $I$ periodically along $\mathbb S^1$ in such a way that the center of $I$ is always in $\mathbb S^1$ and during one period, when the center of $I$ comes back to its initial position, we turn $I$ upside down to obtain the M\"obius band $E$. Here $E$ is an interval bundle, $\mathbb S^1$ being the base space and the positions of $I$ being the fibres over points of $\mathbb S^1$. The described movement, when considering time from $-\infty$ to $+\infty$, gives a flow on $E$ and each time-$t$ map of this flow is a fibre-preserving map on $E$.

We can move $I$ in such a way that for the time-1 map $F$ of the mentioned flow, the restriction  $f=F|_{\mathbb S^1}$ is an irrational rotation, by some angle $\alpha$, of $\mathbb S^1$. Hence $\mathbb S^1$ is a minimal set of $F$. Then the boundary $\partial E$ of $E$ is also a minimal set of $F$, since the restriction of $F$ to $\partial E$ is conjugate to $\alpha/2$ rotation of the circle.

Notice that  the simple closed curve  $\partial E$ is a sub-bundle of $E$ (the fibre having car\-di\-na\-li\-ty~$2$) but it is not a direct product of the base space $\mathbb S^1$ with a two-point set. \hfill $\square$
\end{example}

\begin{example}[\emph{Sturmian minimal sets}]. \label{E:Sturmian}
Consider a Sturmian minimal system $({\frak S}, \sigma)$, see e.g. \cite[pp. 200--203]{V}, satisfying the following properties: it is a minimal subshift of $\{0,1\}^\mathbb Z$ and it is an almost one-to-one extension of a system $(\mathbb S^1; \rot_{\alpha})$, where $\mathbb S^1$ is the circle and $\rot_{\alpha}$ is an irrational rotation. More precisely, if
$\Sigma : ({\frak S}, \sigma) \to (\mathbb S^1; \rot_{\alpha})$ is the corresponding factor map,
then there is a countable dense set $D\subset \mathbb S^1$ such that for all $z\in \mathbb
S^1 \setminus D$ the fibre $\Sigma ^{-1}(z)$ consists of just one point of ${\frak S}$ and for all $z\in D$
the fibre $\Sigma ^{-1}(z)$ consists of two points of ${\frak S}$. We may think of ${\frak S}$ as being a minimal set
of a fibre-preserving map in  $\mathbb S^1 \times [0,1]$, whose base map is $\rot_{\alpha}$. Let us explain this.

The point inverses of $\Sigma$ are the fibres of the mentioned almost 1-to-1 extension and the homeomorphism $\sigma$ sends fibres to fibres. Topologically, ${\frak S}$ is a Cantor set (since the Sturmian system is an uncountable minimal subshift) and so we may assume that ${\frak S}\subseteq [0,1]$. Consider the map $H: {\frak S} \to \mathbb S^1 \times [0,1]$ sending $s \in {\frak S}$ to $(\Sigma (s), s) \in \mathbb S^1 \times [0,1]$. Then $H$ is continuous and injective, so it is an embedding of the set ${\frak S}$ into the cylinder $\mathbb S^1 \times [0,1]$. Moreover, vertical fibres of the Cantor set $H({\frak S})\subseteq \mathbb S^1 \times [0,1]$ correspond to point inverses of $\Sigma$ which means that $H$ induces fibre-preserving dynamics on $H({\frak S})$ which is topologically conjugate to $\sigma$.

Again, by the extension lemma from~\cite{KS1}, one can extend this dynamics on $H({\frak S})$ to a fibre-preserving map $F: \mathbb S^1 \times [0,1] \to \mathbb S^1 \times [0,1]$. Then $H({\frak S})$ is a minimal set of $F$ having singleton fibres with the exception of countably many fibres, each of them consisting of two points. \hfill $\square$
\end{example}

\subsection{Examples of minimal sets with nonempty interior}\label{SS:(A2)}
This case can occur only if  the graph $\Gamma$ contains a circle. As an example,  consider an irrational rotation of the torus ($M$ is the whole torus).
To produce some more general ``direct product" examples  with $B$ being a general compact metric space admitting a minimal map, one can use Proposition~\ref{P:extension} and Corollary~\ref{C:m circles} below.

To prove Proposition~\ref{P:extension}, let us start by recalling a theorem due to H. Weyl (see e.g.~\cite[Chapter I, Theorem~4.1]{KN}) saying that if $(a_n)_{n=1}^\infty$ is a sequence of distinct integers then for almost all (with respect to the Lebesgue measure) real numbers $x$ the sequence $(a_n x)_{n=1}^\infty$ is uniformly distributed modulo 1. As an obvious consequence of this theorem we get that for any sequence of positive integers $n_1 < n_2 < \dots$ there is an angle $\alpha$ such that the rotation $g$ of $\mathbb S^1$ by the angle $\alpha$ is \emph{minimal with respect to the sequence $(n_k)_{k=1}^\infty$}. This means that for every $s\in \mathbb S^1$ the set  $\{g^{n_k}(s) : \, k=1,2,\dots\}$ is dense in $\mathbb S^1$. Of course, any such rotation $g$ is necessarily irrational.

The following simple proposition dealing with direct product maps (rather than with skew product minimal systems as for instance in~\cite{GW}) is, though not most general possible, sufficient for our purposes. We present here a short proof, based on the Weyl's theorem mentioned above.

\begin{proposition}\label{P:extension}
Let $(B,f)$ be a minimal dynamical system, $B$ being a metric space. Then there exists an irrational rotation $g$ of the circle $\mathbb S^1$ such that the direct product system $(B\times \mathbb S^1, f\times g)$ is minimal.
\end{proposition}

\begin{proof}
Fix $x_0 \in B$ and positive integers $n_1 < n_2 < \dots$ such that $f^{n_k}(x_0) \to x_0$ when $k\to \infty$. By the Weyl's theorem, there is an irrational rotation $g$ of $\mathbb S^1$ such that
for every $s\in \mathbb S^1$ the set  $\{g^{n_k}(s) : \, k=1,2,\dots\}$ is dense in $\mathbb S^1$. We claim that $F=f\times g$ is minimal. It is sufficient to prove that the $\omega$-limit set $\omega_F(x,s) = B\times \mathbb S^1$ for every $(x,s)\in B\times \mathbb S^1$.

From the choice of $x_0$ and $g$ it follows that for every $y\in \mathbb S^1$, $\omega_F(x_0,y) \supseteq \{x_0\}\times \mathbb S^1$. Since the $f$-orbit of $x_0$ is dense in $B$ and $F(\omega_F(x_0,y))\subseteq \omega_F(x_0,y)$ and $g$ is surjective, the closed set $\omega_F(x_0,y)$ contains the union of a dense family of fibres. We have thus proved that $\omega_F(x_0,y) = B\times \mathbb S^1$ for every $y\in \mathbb S^1$.

Now fix any point $(x,s)\in B\times \mathbb S^1$. Since $\omega_f(x)=B$ and $\mathbb S^1$ is compact, the set $\omega_F(x,s)$ contains at least one point $(x_0, y)\in \{x_0\}\times \mathbb S^1$. Then $\omega_F(x,s) \supseteq \omega_F(x_0,y) = B\times \mathbb S^1$. \hfill $\square$
\end{proof}

\begin{corollary}\label{C:nove}
Let $E = B\times \Gamma$ be a graph bundle such that $B$ is a compact metric space admitting a minimal map and $\Gamma$ be a graph containing a circle $C$.  Then there exists a fibre-preserving map $F: E\to E$ such that $B\times C$ is a minimal set of $F$. \end{corollary}

\begin{proof}
Using Proposition~\ref{P:extension} extend a minimal map $f: B\to B$ to a minimal map $f\times g : B\times C \to B\times C$.
Then use the fact that there is a retraction $r: \Gamma \to C$ and put $F = f\times (g\circ r)$.  \hfill $\square$
\end{proof}

However, for a general (i.e., not direct product) graph bundle $(E, B, p, \Gamma)$, where $B$ is a compact metric space admitting a minimal map and $\Gamma$ contains a circle, the existence of fibre-preserving maps having  minimal sets with nonempty interior  is not clear at all. For instance, already the construction of  such a minimal homeomorphism on the Klein bottle is not easy, see~\cite{Ell} or~\cite{Parry}. We do not know whether in any graph bundle which is not a tree bundle and whose base admits a minimal map there exists a fibre-preserving map having a minimal set with nonempty interior.

Recall that  $(X,f)$ is a \emph{totally minimal} system if $(X,f^n)$ is minimal for  $n=1,2,\dots$.

\begin{corollary}\label{C:m circles}
Let $(B,f)$ be a totally minimal dynamical system, $B$ being a metric space. Let $\Gamma$ be a graph which  contains $m$ disjoint circles. Denote the union of these circles by $S$. Then there exists a continuous map $h: \Gamma \to \Gamma$ such that $B\times S$ is a minimal set in the direct product system $(B\times \Gamma, f\times h)$.
\end{corollary}

\begin{proof}
Let $g$ be the irrational rotation by angle $\alpha$, which can be assigned to the minimal system $(B,f^m)$ by Proposition~\ref{P:extension}. Fix a circle $C$ in $S$. Let $\widetilde g$ be the map $S\to S$ whose restriction to $C$ is conjugate to $g$ and which is identity on $S\setminus C$. Then compose $\widetilde g$ with a homeomorphism on $S$, which cyclically permutes the $m$ circles in $S$. Finally, extend the selfmap of $S$ obtained in such a way to a continuous selfmap $h$ of $\Gamma$ (this is possible, see e.g. \cite{BHS}). By Proposition~\ref{P:extension}, the set $B\times C$ is minimal for $(f\times h)^m = f^m \times h^m$ since $h^m|_C$ is conjugate to $g$. Then $B\times S$ is minimal for $f\times h$.   \hfill $\square$
\end{proof}

\begin{example}[\emph{Torus attached to the boundary of the M\"obius band as a minimal set}].\label{E:Mobius modified}
We construct a space $E$ similarly as the M\"obius band in Example~\ref{E:Mobius} with only one difference -- now, instead of moving the straight line segment $I$ along the circle $\mathbb S^1$, we move the graph $\Gamma$ which is the segment $I$ with two identical circles attached to $I$ at the endpoints of $I$ in such a way that the intersections of the circles with the straight line segment joining the centers of the circles are the endpoints of $I$. We assume that the diameter of $\Gamma$ is smaller than that of~$\mathbb S^1$. So, $E$ is a M\"obius band whose boundary simple closed curve is replaced by a topological torus $\mathbb T^2$. 

As in Example~\ref{E:Mobius}, we consider the time-1 map $F$ of the flow induced by the mentioned ``movement" and put $f=F|_{\mathbb S^1}$, an irrational rotation of $\mathbb S^1$ by some angle $\alpha$. The map $F$ is fibre-preserving and  we are going to extend it to a fibre-preserving continuous map $G:E\to E$ for which the torus $\mathbb T^2$ will be a minimal set.

Let $\varphi : \Gamma \to \Gamma$ be any continuous map such that the points of $\Gamma$ which are symmetrical with respect to the center of $I$ are mapped to symmetrical points (hence the center of $I$ is a fixed point) and the restriction of $\varphi$
to each of the two circles in $\Gamma$ is an irrational rotation. The symmetry condition requires that both circles rotate by the same angle $\beta$ and with the same ``orientation". Further, let $\Phi: E\to E$ be a continuous map which maps each of the fibres of $E$ into itself in such a way that the restriction of $\Phi$ to each of the fibres is an isometric copy of $\varphi$ (the fibres of $E$ are isometric to $\Gamma$). Simply, in one of the fibres we choose an orientation of the circles (the same orientation), hence also the ``orientation" of the $\beta$-rotations on them. The continuity of $\Phi$ then determines the ``orientation" of the rotations on the circles in all other fibres. (Since we have the same orientation of the circles in $\Gamma$, one can see that this is a correct construction, we really get a well defined map $\Phi$.)

Put $G=  \Phi \circ F$. Then $G$ is a fibre-preserving map on the graph-bundle $E$ and the restriction of $G$ to the torus $\mathbb T^2$ is a double rotation -- irrational $\alpha/2$-rotation in one direction and $\beta$-rotation in the other direction. Now we restrict ourselves to  $\beta$ for which $G$ is a minimal map on $\mathbb T^2$. 
Notice that, in contrast to Corollary  \ref{C:m circles},  the obtained minimal set   $\mathbb T^2$  is not a direct product of the base space $\mathbb S^1$ with a union of  circles. \hfill $\square$
\end{example}

\subsection{Proof of Theorem D}\label{SS:(D)}

Given a set $A\subseteq \mathbb R^k$ and a vector $v\in \mathbb R^k$, by $A+v$ we mean the set $\{a+v:\, a\in A\}$. 

 \vspace{2mm}

\noindent {\bf Theorem D.} {\it
 There are a minimal selfmap $f$ of a Cantor set $B$, a connected  graph $\Gamma$ and an extension $(B\times \Gamma, F)$ of $(B,f)$ with  a minimal set $M$ such that, for some $b\in B$, 
\begin{itemize}
\item[$\bullet$] $M_z$ is a circle for each $z \not=b$, and  
\item[$\bullet$] $M_b$ is a union of two circles. 
Depending on the choice of such a system,  the union of any two different circles in any graph 
can appear as the set $M_b$. 
\end{itemize}}

\begin{proof} 

\underline{\it Case I:}\ {\it  $M_b$ is union of two disjoint circles}.

Let $(C,f)$, with $C$ being a subset of the real line, be a Cantor minimal system such that one point has two pre-images
and all the other points have only one pre-image each. Such systems appear for instance in symbolic  and interval dynamics.  It will be convenient to give an explicit construction of such a system 
in order to introduce  the notation which will be used throughout the whole proof.  Start with the dyadic adding machine on the Cantor ternary set. Recall that it is often viewed as a restriction of an interval map to the invariant Cantor set, usually a restriction of the map shown for instance in~\cite[Fig. 1]{SS}; notice that then the adding machine is increasing at each point except at the rightmost one where it is decreasing. Choose a point $a$ in this Cantor set which does not belong to the countable set consisting of the endpoints of the contiguous intervals (including the leftmost and the rightmost points of the Cantor set). Hence the points $a_{-j}:= f^{-j}(a)$, $j=1,2,\dots$ do not belong to this countable set, too. Now blow up the backward orbit of $a$,  i.e.,  for $j=1,2,\dots$, replace the point $a_{-j}$ by a compact interval with length $L_{-j}$ with convergent sum $\sum _{j=1}^{\infty} L_{-j}$ and remove the interior of this interval. This means that the points $a_{-j}$, $j=1,2,\dots$ are ``doubled", i.e. replaced by pairs of points $a_{-j}^- < a_{-j}^+$. What we get is again a Cantor set. Consider the dynamics on it which is inherited from the adding machine, except for the ``new" points $a_{-j}^-, a_{-j}^+$, $j=1,2,\dots$ where we still need to define the dynamics. To this end, send both $a_{-1}^-$ and $a_{-1}^+$ to $a$ and, since the adding machine is increasing at each $a_{-j}$ and we want a continuous dynamics, for $j=2,3,\dots$ send $a_{-j}^-$ to $a_{-j+1}^-$ and $a_{-j}^+$ to $a_{-j+1}^+$. The map defined in such a way is continuous and the system is minimal.

Recall that, up to a homeomorphism, there is only one Cantor set and it is homogeneous. Therefore, no matter which of the Cantor minimal systems $(C,f)$ (such that one point has two pre-images and all the other points have only one pre-image) we choose, we may think of $C$ as a Cantor set on the real line, with the point having two pre-images being for instance the rightmost point of~$C$. For the same reason we can also assume that the two-preimages, denote them $c_l < c_r$, are the endpoints of a contiguous interval (this is important for geometry of our construction below).

Applying Proposition~\ref{P:extension} we
extend $(C,f)$ to a minimal system $(C\times S_1, f\times g)$ where
$g$ is an irrational rotation of the circle $S_1 = \{(y,z)\in
\mathbb R^2:\, y^2 +z^2 = 1\}$. Denote by $a_1$ and $b_1$ the
$g$-images of the points $(0,1)$ and $(0,-1)$, respectively. Let
$J_1$ be one of the half-circles determined by $a_1, b_1$.

The set $C$ is the union of $C_L = \{x\in C:\, x\leq c_l\}$ and $C_R = \{x\in C:\, x\geq c_r\}$.
Put $C_R^{-} = C_R - (c_r-c_l)$. Then $C_L \cup C_R^{-}$ is a Cantor set with
$C_L \cap C_R^{-} = \{c_l\}$. Further put $S_2 = S_1 + (0,3)$, $a_2 = a_1 + (0,3)$, $b_2 = b_1 + (0,3)$
and $J_2 = J_1 + (0,3)$. Finally, denote $M= (C_L \times S_1) \cup (C_R^{-} \times S_2)$. The dynamical
system $(C\times S_1, f\times g)$ induces in a natural way a (minimal) dynamical system $(M,F)$ which
is topologically conjugate to $(C\times S_1, f\times g)$ and is obtained from $(C\times S_1, f\times g)$
by just replacing $(C_R\times S_1)$ by its translate $(C_R^{-}\times S_2)$, `without changing dynamics'.
In the new system $(M,F)$ the map $F$ preserves `vertical' fibres; the fibre over $c_l$ consists of two
circles, each of the other fibres is just a circle. Denote by $\varphi$ the base map of $F$.
It is clear that $(M,F)$ can be considered as a minimal extension of the
dynamical system $(C_L \cup C_R^{-}, \varphi)$ obtained from $(C,f)$ by identifying
points $c_l$ and $c_r$. Let
$\Gamma = S_1 \cup I \cup S_2$ where $I \subseteq \mathbb R^2$ is the `vertical' interval with
end-points $(0,1)$ and $(0,2)$. Put $E=(C_L\cup C_R^{-}) \times \Gamma$. Then $\Gamma$ is a connected
graph and $E$ is a graph bundle with fibre $\Gamma$.

We claim that the map $F$ can be extended to a continuous fibre-preserving map $G: E\to E$.
We are going to define $G$. Of course, $G|_{M} = F$. Further, for $x\in C_L \setminus \{c_l\}$
and $(y,z)\in S_2$ put $G(x,y,z) = F(x, y, z-3)$ and for $x\in C_R^{-} \setminus \{c_l\}$ and $(y,z)\in S_1$
put $G(x,y,z) = F(x, y, z+3)$. So, $G$ is already defined on $(C_L\cup C_R^{-})\times (S_1 \cup S_2)$.
It remains to define $G$ on $(C_L\cup C_R^{-}) \times (I\setminus \{(0,1), (0,2)\})$. So, fix $x\in C_L\cup C_R^{-}$. Then
$G(\{x\}\times (S_1\cup S_2)) = \{\varphi(x)\}\times S_i$ for some $i\in \{1,2\}$. Further,
$G(x,0,1) = \{\varphi(x)\}\times \{a_i\}$ and $G(x,0,2) = \{\varphi(x)\}\times \{b_i\}$. For $1<z<2$
let $G(x,0,z)$ be the point of $\{\varphi(x)\}\times J_i$ such that the length of the sub-arc of
$\{\varphi(x)\}\times J_i$ with end-points $\{\varphi(x)\}\times \{a_i\}$ and $G(x,0,z)$ equals $\pi \cdot (z-1)$.

Then $G$ maps $E$ continuously onto its unique  minimal set $M$. Here $M_{c_l}$ is the
union of two circles and $M_b$ for $b\neq c_l$ is a circle. So, $M$ is not a sub-bundle of $E$. 

\underline{\it Case II:}\ {\it  $M_b$ consists of two arbitrarily intersecting circles whose union is a graph}. 

Before giving such a construction we wish to mention that if $E$ were not required to be a graph bundle, it would be sufficient to consider a skew product minimal map on the pinched torus from~\cite{BKS}. In that example, one fibre is ``figure eight" (two circles intersecting in one point), all the other fibres are circles (simple closed curves).

The union  $P\cup Q$ of disjoint sets  will sometimes be denoted by $P\sqcup Q$. We will also keep   the notations from {Case I.}  
Starting with the minimal system $(C\times S_1, f\times g)$ we are going to produce a fibre-preserving selfmap $G^*$ of a direct product graph bundle $E^*\subseteq \mathbb R^3$ with the following properties 
\begin{enumerate}
\item $E^*= (C_L\cup C_R^{-})\times \Gamma^*$
\item $\Gamma^* = S_1 \cup S_1^*$ where $S_1$ is the ``geometrical" circle  $y^2+z^2=1$ and $S_1^*$ is a ``topological" circle (i.e., a simple closed curve) such that
    \begin{itemize}
    \item[$\bullet$] $\emptyset \neq S_1\cap S_1^*\neq S_1$ has finitely many connected components (just because we want $E^*$ to be a graph bundle, i.e. $\Gamma^*$ has to be a graph),
    \item[$\bullet$] $S_1^*$ is a subset of the closed disc bounded by the circle $S_1$ and each radius of $S_1$ contains exactly one point of $S_1^*$.
    \end{itemize}
\item $M^*= (C_L \times S_1) \cup (C_R^{-} \times S_1^*)$ is a minimal set for $G^*$.
\end{enumerate}
Note that each fibre of $M^*$ consists of one circle, except of $M^*_{c_l}$ which consists of two intersecting circles $\{c_l\}\times \{S_1\}$ and $\{c_l\}\times \{S_1^*\}$. Though the only restrictions for the choice of $S_1^*$ are those in~(2), let us explicitly mention three simplest cases:
\begin{enumerate}
\item[($12_1$)] $S_1$ and $S_1^*$ intersect just in one point (hence $M^*_{c_l}$ is homeomorphic to the ``figure eight"), or
\item[($12_2$)] $S_1$ and $S_1^*$ intersect in an arc ($M^*_{c_l}$  is homeomorphic to the ``figure $\Theta$"), or
\item[($12_3$)] $S_1$ and $S_1^*$ intersect in two points.
\end{enumerate}

\smallskip

So, we have the minimal system $(C\times S_1, F_1)$, where $C=C_L \sqcup C_R$ is a Cantor set on the $x$-axis with $\max C_L = c_l < c_r = \min C_R$ and $F_1 = f\times g$. We are going to construct $(E^*, G^*)$ as above. Fix $S_1^*$ as in (2). Denote by $\alpha$ the projection of $S_1$ onto $S_1^*$ along the radii of $S_1$ (hence $\alpha$ is identity on $S_1\cap S_1^*$) and by $\sigma$ the map
$$
\sigma (x,y,z):= \begin{cases} (x,y,z), & \text{if $(x,y,z) \in C_L\times S_1$}, \\
                               (x, \alpha (y,z)), & \text{if $(x,y,z) \in C_R\times S_1$}.
      \end{cases}
$$
Then $\sigma: C\times S_1 \to (C_L \times S_1) \sqcup (C_R \times S_1^*)$ is a homeomorphism and so the map
$$
F_1^*:= \sigma \circ F_1 \circ \sigma^{-1},
$$
being topologically conjugate to $F_1$, is a continuous \emph{minimal} selfmap
of $(C_L \times S_1) \sqcup (C_R \times S_1^*)$. Since $f(c_l)=f(c_r) = \max C$, the definition of $F_1^*$ gives that
\begin{equation}\label{E:prve}
F_1^*(\{c_l\}\times S_1) = \{\max C\} \times S_1^* = F_1^*(\{c_r\}\times S_1^*),
\end{equation}
\begin{equation}\label{E:druhe}
F_1^*(c_l, y, z)  =  \sigma (F_1 (c_l, y,z)) = \sigma (F_1 (c_r, y,z)) = F_1^*(c_r, \alpha(y,z)) \ \text{for $(y,z)\in S_1$.}
\end{equation}
Then~(\ref{E:prve}) and~(\ref{E:druhe}) imply that
\begin{equation}\label{E:2 kruznice rovnako}
F_1^*(c_l,y,z) = F_1^*(c_r,y,z) \in \{\max C\}\times S_1^*, \quad \text{ for } (y,z)\in S_1\cap S_1^*.
\end{equation}
Now the idea is to identify the pairs of points
$(c_l,y,z), (c_r,y,z)$ where  $(y,z)\in S_1\cap S_1^*$
 with the same $F_1^*$-images and to produce in such a way a new map $F_2^*$ on a new space $M^*$. Since we wish to keep under control the geometry of our example, we proceed geometrically. In view of~(\ref{E:2 kruznice rovnako}), the mentioned pairs of points  become identified if we translate $C_R\times S_1^*$ by the vector $(-(c_r-c_l),0,0)$. Therefore denote $M^*:=(C_L\times S_1) \cup (C_R^{-}\times S_1^*)$ and let $T: (C_L \times S_1) \sqcup (C_R \times S_1^*)  \to M^*$ be defined by
\begin{equation}
T(x,y,z):= \begin{cases}
              (x,y,z),                  & \text{if $(x,y,z) \in C_L\times S_1$}, \\
              (x -(c_r-c_l),y,z),       & \text{if $(x,y,z) \in C_R\times S_1^*$}. \notag
            \end{cases}
\end{equation}
As already indicated, due to~(\ref{E:2 kruznice rovnako}) there is a unique continuous map $F_2^*: M^*\to M^*$ such that the following diagram commutes:
\[
\begin{CD}
(C_L \times S_1) \sqcup (C_R \times S_1^*) @>{F_1^*}>> (C_L \times S_1) \sqcup (C_R \times S_1^*)\\
@V{T}VV @VV{T}V \\
M^* @>{F_2^*}>> M^*
\end{CD}
\]

\smallskip

\noindent A straightforward analysis of the map $F_2^*$ shows that 
a point in $\{c_l\}\times S_1\subseteq M^*$ and a point in $\{c_l\}\times S_1^*\subseteq M^*$ lying on the same radius of the circle $\{c_l\} \times S_1$ have always the same $F_2^*$-image:
\begin{equation}\label{E:images of friends}
F_2^*(c_l, y, z) = F_2^*(c_l, \alpha(y, z)) \quad \text{ whenever } (y,z)\in S_1.
\end{equation}
To finish the study of the properties of $F_2^*$, notice that $F_2^*$ is fibre-preserving and, being a factor of the minimal map $F_1^*$, is also minimal.

Now define $E^*$ and $\Gamma^*$  as in 1. and 2.  at the beginning of the proof of Case II.  To finish our construction, it is sufficient to extend $F_2^*: M^*\to M^*$ to a continuous fibre-preserving map $G^*: E^*\to E^*$. Here is one such extension:
\begin{equation}\label{E:G hviezdicka}
G^*(x,y,z):= \begin{cases} F_2^*(x,y,z), & \text{if $(x,y,z)\in M^*$}, \\
                           F_2^*(x,\alpha^{-1}(y,z)), & \text{if $(x,y,z) \in (C_L\setminus \{c_l\}) \times S_1^*$}, \\
                           F_2^*(x,\alpha(y,z)), & \text{if $(x,y,z) \in (C_R^{-}\setminus \{c_l\}) \times S_1$}.
      \end{cases}
\end{equation}
The definition is correct. In fact, the first and the second case are compatible, because if $(x,y,z)\in M^*$ and simultaneously $x\in C_L\setminus \{c_l\}$ and $(y,z)\in S_1^*$, then $(y,z)\in S_1\cap S_1^*$ and so $(\alpha^{-1})(x,y,z) = (x,y,z)$. Analogously, the first and the third case are compatible. Hence, $G^*:E^*\to E^*$ is a well defined extension of $F_2^*$. It is obviously fibre-preserving. To show that it is continuous, it is sufficient to show that the restrictions of $G^*$ to the \emph{closed} sets $C_L\times (S_1\cup S_1^*)$ and $C_R^{-}\times (S_1\cup S_1^*)$ are continuous. Since the arguments for both cases are analogous, we prove only the continuity of $G^*$ on the former set. It is the union of two \emph{closed} sets $C_L\times S_1$ and $C_L\times S_1^*$ and so the continuity of $G^*|_{C_L\times (S_1\cup S_1^*)}$ follows from the following two facts:
\begin{itemize}
\item[$\bullet$] On the set $C_L\times S_1$, since it is a subset of $M^*$, the map $G^*$ is continuous because it coincides there, by~(\ref{E:G hviezdicka}), with the continuous map $F_2^*$.
\item[$\bullet$]  On the set $C_L\times S_1^*$ the map $G^*$ is also continuous, because it coincides there with the continuous map $F_2^*\circ (\id_{C_L} \times \alpha^{-1})$ where $\id_{C_L}$ is the identity on $C_L$. To see this, first notice that for $x\in C_L\setminus \{c_l\}$ the coincidence works by~(\ref{E:G hviezdicka}). Further, if $(y^*,z^*)\in S_1^*$ then $(c_l, y^*,z^*)\in M^*$ and so, using~(\ref{E:G hviezdicka}) and~(\ref{E:images of friends}) we get $G^*(c_l, y^*,z^*) = F_2^*(c_l,y^*,z^*)  = F_2^*(c_l,\alpha^{-1}(y^*,z^*))$, as required.
\end{itemize}

The construction is now completed. In the case ($12_1$)  it gives  the space $E^*$ made of two ``tubes" $(C_L \cup C_R^{-})\times S_1$ and $(C_L \cup C_R^{-})\times S_1^*$, the second tube lying ``inside" the first one and so they touch ``internally". If one wishes that they touch ``externally", i.e. that $M^*_{c_l}$ is a geometric, not only topological ``figure eight", it is sufficient to use an appropriate conjugacy. Similarly, in ($12_2$) the tubes can "touch externally" along an arc. Also in ($12_3$) we can get that $S_1^*$ is not anymore a subset of the closed disc bounded by the circle $S_1$, but $S_1$ and $S_1^*$ are two geometric circles having two points in common. \hfill $\square$
\end{proof}


\section{Dynamical and topological preliminaries}\label{S:prelim}

For convenience of the reader, we collect below 
several dynamical and topological facts which will be used throughout the rest of the paper. 
The reader should at least pay attention to the concepts of a redundant open set and 
the homeo-part of a minimal system since they are instrumental in the paper.

\subsection{Some basic facts on minimality}\label{SS:minimality}

In this subsection we always assume that $X$ is a compact metric space and $f:X\to X$ is a continuous map.
The facts here, if not obvious, are mostly results from our paper~\cite{KST}. An exception is the equivalence (1) $\Leftrightarrow$ (3) in the below list of equivalent definitions of minimality, which is~\cite[Lemma 3.10]{B}. For the equivalence (1) $\Leftrightarrow$ (2) involving backward orbits one needs to see the proof of Theorem 2.8 in~\cite{KST} (cf.~\cite{Mal}).  

In the introduction we gave two equivalent definitions of minimality (in terms of invariant subsets and in terms of density of forward orbits). For a compact metric space $X$ and a continuous map $f: X\to X$ also the following are equivalent:
\begin{enumerate}
\item[(1)] $(X,f)$ is minimal,
\item[(2)] $f(X)=X$ and every backward orbit of every point in $X$ is dense (by a \emph{backward orbit} of $x_0\in X$ we mean any set $\{x_0, x_1, \dots, x_n, \dots \}$ with $f(x_{i+1})=x_i$ for $i\geq 0$),
\item[(3)] the only closed subsets $A$ of $X$ with $f(A) \supseteq A$ are $\emptyset$ and $X$,
\item[(4)] for every non-empty open set $U\subseteq X$, there is $N\in \mathbb N$ such that $\bigcup_{n=0}^N f^{-n}(U)=X$.
\end{enumerate}

\noindent We will also need some necessary conditions for minimality. If $(X,f)$ is minimal then
\begin{enumerate}
\item [(a)] for every non-empty open set $U\subseteq X$, there is $N\in \mathbb N$ such that $\bigcup_{n=0}^N f^{n}(U)=X$,
\item [(b)] $f$ is \emph{feebly open}, i.e. it sends non-empty open sets to sets with non-empty interior,
\item [(c)] $f$ is \emph{almost one-to-one}, which means that the set $\{x\in X:\, \card f^{-1}(x) = 1\}$ is a $G_{\delta}$-dense set in $X$,
\item [(d)] if $A\subseteq X$ is nowhere dense (dense, of 1st category, of 2nd category, residual) then both $f(A)$ and $f^{-1}(A)$ are nowhere dense (dense, of 1st category, of 2nd category, residual), respectively.
\end{enumerate}

A set $G\subseteq X$ is said to be a \emph{redundant open set for a map $f: X \to
X$} if $G$ is nonempty, open and $f(G)\subseteq f(X\setminus G)$
(i.e., its removal from the domain of $f$ does not change the image
of $f$).  

\begin{lemma}[\cite{KST}]\label{L:zbytocna}
Let $X$ be a compact metric space and $f: X \to X$ continuous.
Suppose that there is a redundant open set for $f$. Then the system $(X,f)$ is not minimal.
\end{lemma}

\subsection{Homeo-part of a minimal system}\label{SS:homeo-part}

\begin{definition}\label{D:homeo-part}
Let $f$ be a continuous selfmap of a compact metric space $X$. Let $H\subseteq X$ be the set of all points
$x_0\in X$ whose full orbit 
$
\left\{x \in X:\, \exists i,j\geq 0 \text{ with } f^i (x) = f^j (x_0) \right\}
$
is of the form $\left\{ \dots ,x_{-2}, x_{-1}, x_0, x_1, x_2, \dots \right \}$ where $f(x_n)=x_{n+1}$ for
every integer $n$. Then the system $(H,f|_H)$ is said to be the \emph{homeo-part of the system $(X,f)$}.
We also shortly say that $H$ is the homeo-part of $f$.
\end{definition}
One can show that $H$ is always a $G_{\delta}$ set (possibly empty).  For minimal maps this is easier to prove and we can say even more. 
\begin{lemma}\label{L:homeo-part is G-delta}
Let $X$ be a compact metric space and $f: X\to X$ be a minimal map.  Then the homeo-part 
 $H$ of $f$  is a dense $G_{\delta}$ set.
\end{lemma}

\begin{proof} Set $D= \{x\in X: \card f^{-1}(x) > 1\}$. By~\cite[Theorem~2.8]{KST}, the homeo-part of a minimal map is residual and $D$  is an $F_\sigma$-set of first category.  It is straightforward to check that $H = X\setminus \bigcup_{n=-\infty}^{+\infty} f^n(D)$. By (d) from Subsection \ref{SS:minimality}  we get that  $H$ is $G_\delta$. 
  \hfill 	$\square$
\end{proof}

\begin{lemma}\label{L:noninvertibility f inverse}
Let $X$ be a compact metric space and $f: X\to X$ be a continuous minimal map. Let a set $D=\{ \dots ,x_{-2}, x_{-1}, x_0, x_1, x_2, \dots \}$ be such that $f(x_n)=x_{n+1}$ for every integer $n$ (i.e., $D$ is a union of the forward orbit of $x_0$ and one of the backward orbits of $x_0$). Suppose that there is a point in $D$ that has more than one $f$-preimage in $X$ (or, equivalently, an $f$-preimage in $X\setminus D$). Then $(f|_{D})^{-1}$ is not continuous. 
\end{lemma}

\begin{proof}
Suppose, on the contrary, that $g:=(f|_{D})^{-1}: D\to D$ is continuous. Without loss of generality we may assume that the mentioned point with two preimages is $x_0$. Denote by $z$ a point in $X\setminus D$ with $f(z)=x_0$. Choose two disjoint open neighborhoods $U_{-1}$ and $U_z$ of the points $x_{-1}$ and $z$, respectively. Denote $V_{-1}:= U_{-1}\cap D$. Due to the continuity of $g$ at the point $x_0$, we can find a neighborhood $U_0$ of $x_0$ such that for $V_0:=U_0\cap D$ we have $g(V_0)\subseteq V_{-1}$. Now use the continuity of $f$ at the point $z$ to get an open neighborhood $U_z^*\subseteq U_z$ of $z$ with $f(U_z^*)\subseteq U_0$. Since $f$ is minimal, there is $k>0$ with $x_k \in U_z^*$, whence $x_{k+1}=f(x_k) \in V_0$. Then $g(x_{k+1}) = x_k \in U_z^*$ which contradicts the facts that $g(V_0)\subseteq V_{-1}$ and $U_z^*$ is disjoint with $V_{-1}$. \hfill 	$\square$
\end{proof}

The next description of properties of the homeo-part of a minimal map follows partially from Theo\-rem~2.8 and its proof in~\cite{KST}. Note that the notion of a full orbit of a point (for a not necessarily invertible map) was introduced in Definition~\ref{D:homeo-part}.

\begin{lemma}\label{L:homeo-part properties}
Let $X$ be a compact metric space and $f: X\to X$ be a continuous minimal map. Let $H\subseteq X$ be the homeo-part of $f$.
Then: 
\begin{enumerate}
\item [(1)]$f(H)=H=f^{-1}(H)$ (equivalently, $H$ is a union of full orbits of the map $f$),
\item [(2)]every point of the set $H$ has just one $f$-pre-image (and this pre-image lies in $H$),
\item [(3)]both $f|_H$ and $(f|_H)^{-1}$ are minimal homeomorphisms $H\to H$,
\item[(4)] $H$ is a $G_{\delta}$ dense subset of $X$,
\item [(5)] $H$ is a \emph{maximal} subset of $X$ with the properties (1)~and~(2),
\item [(6)] $H$ is a \emph{maximal} subset of $X$ with the property (3).
\end{enumerate}
\end{lemma}

\begin{proof}
The equivalence in (1) is obvious. The properties (1) and (2) follow from the definition of the homeo-part, see Definition~\ref{D:homeo-part}. For the property (3) see Theo\-rem~2.8 in~\cite{KST} and its proof. Lemma~\ref{L:homeo-part is G-delta} gives (4). The property (5) is obvious since if we add something to $H$, we have to add another full orbit (because we want (1)). This full orbit contains, due to the definition of the homeo-part, a point with two preimages. Then the enlarged set will not satisfy (2). Similarly, Lemma~\ref{L:noninvertibility f inverse} shows that if we add something to $H$ then the enlarged set will not satisfy (3) and so we get (6). \hfill 	$\square$
\end{proof}

\begin{lemma}\label{L:homeo-part in P}
Let $X$ be a compact metric space and $f: X\to X$ be minimal. Let  $H$ be the homeo-part of $f$ and $P$ be a residual set in $X$. Then there is a set $R\subseteq X$ such that
\begin{enumerate}
\item[(1)] $R\subseteq P\cap H$  and $R$ is residual in $X$,
\item[(2)] $f(R)= R = f^{-1}(R)$,
\item[(3)] both $f|_R$ and $(f|_R)^{-1}$ are minimal homeomorphisms $R\to R$.
\end{enumerate}
In particular, the inclusion $R\subseteq H$ and (2) give that $R$ is a union of some of the full (i.e. forward and backward) orbits of the homeomorphism $f|_H$.
\end{lemma}

\begin{proof} Put $R = H \cap \bigcap_{n\in \mathbb Z} f^n(P)$.
The minimal map $f$ preserves residuality in both forward and backward direction. Therefore the set $R$, being the intersection of countably many residual sets, is residual in $X$. The rest is obvious. \hfill 	$\square$
\end{proof}

\subsection{Locally closed sets and a generalization of Baire category theorem}\label{SS: BCT}

A subset $S$ of a topological space $X$ is \emph{locally closed} if every $x\in S$ has a neighborhood $U$  such that the intersection $S\cap U$ is closed in the subspace $U$ of $X$. The following conditions are equivalent, see e.g. \cite[p. 112]{Eng}:
\begin{enumerate}
\item[(1)] The set $S$ is locally closed.
\item[(2)] The difference $\overline {S} \setminus S$ is closed (i.e. $S$ is open in $\overline{S}$). 
\item[(3)] $S$ is a difference of two closed sets (intersection of a closed set with an open set).
\end{enumerate}

\begin{lemma}\label{L:locally closed}
Let $X$ be a topological space and $S\subseteq X$ a locally closed set. If $S$ is not nowhere dense then $S$ has nonempty interior.
\end{lemma}

\begin{proof}
By the assumption (if $\Inte_X$ denotes the interior in $X$) we have $V:=\Inte_X(\overline{S})\neq \emptyset$. The set $V$ is open in $X$ and hence, being a subset of $\overline{S}$, obviously also open in $\overline{S}$. Now the fact that $S$ is dense in $\overline{S}$ gives that $S\cap V \neq \emptyset$. Further, since $S$ is locally closed, $S$ is open in $\overline{S}$. Therefore there is a set $U$ open in $X$ such that $S= U\cap \overline{S}$. Since $S\subseteq U$ and $S\cap V \neq \emptyset$, we have $U\cap V \neq \emptyset$. This set is open in $X$ and since $U\cap V \subseteq U\cap \overline{S} = S$ we get that $\Inte_X(S)\neq \emptyset$. \hfill 	$\square$
\end{proof}

Recall that a \emph{Baire space} is a topological space having the property that whenever a countable union of closed sets has nonempty interior then one of them has nonempty interior (i.e. so called Baire category theorem works). The following lemma gives a generalization of Baire category theorem: it shows that closed sets can be replaced by locally closed ones.

\begin{lemma}\label{L:Baire} Let $X$ be a Baire topological space and $\{S_\lambda:\ \lambda \in \Lambda\}$ a countable family of subsets of~$X$. Assume that
\begin{enumerate}
\item [(i)] $\bigcup_{\lambda \in \Lambda} S_\lambda$ has nonempty
interior in $X$ and
\item [(ii)] for every $\lambda \in \Lambda$, $S_\lambda$ is locally closed.
\end{enumerate}
Then there is $\lambda_0 \in \Lambda$ such that $S_{\lambda_0}$ has
nonempty interior in $X$.
\end{lemma}

\begin{proof} By applying Baire category theorem to the closed sets
$\overline{S_\lambda}$, $\lambda \in \Lambda$, we get that there is
$\lambda_0 \in \Lambda$ such that $\overline{S_{\lambda_0}}$ has
nonempty interior. So, $S_{\lambda_0}$ is not nowhere dense and since it is locally closed, it has nonempty interior in $X$ by Lemma~\ref{L:locally closed}.  \hfill 	$\square$
\end{proof}


\section{Strongly star-like interior points}\label{S:slip}


We introduce the notion of a \emph{strongly} star-like interior point which
is more restrictive than that of a star-like interior point of $M$ and,
though not appearing in the statement of Theorem~A, will play a key role
in the proof of it.

First of all recall that, when speaking on a graph bundle, we always assume
that it is a (compact) \emph{metric space}, as it was already said in Introduction.
To avoid cumbersome formulations, we will often make no distinction
between homeomorphic spaces. If $(E, B, p, \Gamma)$ is a graph
bundle  and $Q\subseteq E$ and $Z\subseteq \Gamma$, then we say that $Q$
is \emph{canonically homeomorphic} to $U\times Z$, if $p(Q)=U$ and there is a homeomorphism $h:
Q\to U\times Z$ such that on $Q$ we have $\pr_1 \circ h = p$ (here $h$ is said to be a
\emph{canonical homeomorphism}). Notice that, in this terminology, in the above definition of the fibre
bundle it is required that $p^{-1}(U)$ be canonically homeomorphic to $U\times \Gamma$.

Recall that if $(E, B, p, \Gamma)$ is a graph bundle and $M\subseteq E$ and $b\in B$, then
the the fibre of $M$ over $b$ is $M_b = M \cap p^{-1}(b)$. Further, by $\Gamma_b$ we will denote
the set $p^{-1}(b)$, the fibre over $b$ (now we slightly abuse the already adopted notation $M_b$,
since $\Gamma$ is not a subset of $E$). Note that $\Gamma_b = E_b \subseteq E$ is a graph
homeomorphic to $\Gamma$ and if $E=B\times \Gamma$ then $\Gamma_b = \{b\} \times \Gamma$.
Also subsets of $\Gamma_b$ will be sometimes denoted by, say, $I_b$, $T_b$, etc. We believe that this
will not cause any misunderstanding because always when using notation like $X_b$ it will be clear
what kind of a set it is. Recall also that if $M\subseteq E$ and $U\subseteq B$, we denote $M_U = M \cap p^{-1}(U)$.

By an arc we mean a homeomorphic image of a compact real interval.
Sometimes we call it a closed arc, since in an obvious way we also
use the notions of an open or a half-closed arc.  For
$N\geq n \geq 2$ let $\Sigma_n \subseteq \Sigma_N$ be two open stars
with the same central point. Suppose that $\Sigma_n$ is the union of
some of the half-closed branches of $\Sigma_N$ (i.e., $\Sigma_n$ is
obtained from $\Sigma_N$ by removing $N-n\geq 0$ open branches of
$\Sigma_N$). Then we will say that $\Sigma_n$ is a \emph{full
sub-star of $\Sigma_N$}. Here `full' does not mean that $n=N$; it
just refers to the fact that $\Sigma_n$ consists of `whole' branches
of $\Sigma_N$ (rather than of just subsets of them) and so it can be
$n <N$. Note also that we consider only the case when $N\geq n \geq
2$ (though, formally, such a definition would make sense for
$N\geq n \geq 1$).

\begin{definition}\label{D:strongly star-like}
Suppose that $M$ is a closed subset of a \emph{product} graph bundle $E=B\times
\Gamma$. Then we define $\Lint_s (M)$, the set of strongly star-like interior points of $M$, as follows.
A point $x=(x_1,x_2) \in M$ is said to be a \emph{strongly star-like interior point of~$M$}, if
\begin{itemize}
\item[$\bullet$] $x$ has order $N\geq 2$ in the graph $\Gamma_{x_1} = \{x_1\}\times \Gamma$ (so, $\ord (x_2, \Gamma) = N \geq 2$), and
\item[$\bullet$]  there exists an $E$-open neighborhood $O\times \Sigma_N$ of $x$ such that $x_2$ is the central
point of $\Sigma_N$ and the corresponding $M$-open neighborhood $\mathcal{G}= M \cap {(O\times \Sigma_N)}$
of $x$ has the following structure:
\smallskip
   \begin{quotation}
      \noindent  $\mathcal{G}_{x_1} = \{x_1\} \times \Sigma_{k}$ where $k\geq 2$ and $\Sigma_k$ is a full sub-star
       of $\Sigma_N$, and for every $z \in p(\mathcal G)\subseteq O$ we have
       $\mathcal{G}_z = \{z\} \times \Sigma_{k(z),z} \subseteq \{z\} \times \Sigma_k$, where
       $k(z) \in \{2, \dots, k\}$ and $\Sigma_{k(z),z}$ is a full sub-star of $\Sigma_k$.
       (Notice that $\Sigma_{k(x_1),x_1} = \Sigma_k$.) We will say that $\mathcal{G}$ is a
       \emph{canonical $\Lint_s(M)$-neighborhood of~$x$} (note that, among others, $\mathcal{G} \subseteq  \Lint_s (M)$).
   \end{quotation}
\smallskip
\end{itemize}

Above, $\Lint_s (M)$ was defined for a closed subset $M$ of $E=B\times
\Gamma$. Since each graph bundle is locally trivial and the
above definition has a local character, the concept of a strongly star-like interior point  has an obvious extension to the case when the graph bundle $E$
is not a direct product space. For a closed set $M$ in an arbitrary graph
bundle we set $\End_s(M)= M\setminus \Lint_s (M)$.
\end{definition}
\begin{example}\label{Ex:examplepredA}
Let $E=B\times \Gamma$ where $B=[ 0,1]$ and $\Gamma = ([ -1,1]
\times \{0\}) \cup (\{0\}\times [ 0, 1])$. Put $A=[ 0,1] \times [
-1,1] \times \{0\}$ and
\begin{align}
M^1 & = A \cup \{(x,0,x):\, x\in [ 0,1]\}, \quad 
M^2  = M^1 \cup \{(0,0,z):\, z\in [ 0,1]\},\notag \\
M^3 & = A \cup \{(0,0,z):\, z\in [ 0,1]\}, \quad
M^4  = A \cup \{(x,0,1-x):\, x\in [ 0,1]\}~. \notag
\end{align}
Then $(0,0,0) \notin \Lint_s(M^i)$ for $i=1,2$ and $(0,0,0) \in \Lint_s(M^i)$ for $i=3,4$. \hfill $\square$
\end{example}

In the definition we write $\Sigma_{k(z),z}$ rather than $\Sigma_{k(z)}$ because it may happen
that $\Sigma_{k(z_1),z_1}$ and $\Sigma_{k(z_2),z_2}$, considered as subgraphs of $\Gamma$,
are different even when $k(z_1)= k(z_2)$. The following instructive example illustrates this fact.

\begin{example}\label{Ex:exampleA}
Let $E=B\times S_4$ where $B=[ 0, 1 ]$. Let $(C_n)_{n=1}^\infty$ be
a sequence of pairwise disjoint Cantor sets in $(0, 1]$ converging,
in the Hausdorff metric, to the singleton $\{0\}$. Denote three of
the four closed branches of $S_4$ by $J_1,J_2,J_3$ and the central
point of $S_4$ by $c$. Let $M$ be the set with
$$
M_x = \begin{cases} \{x\}\times (J_1 \cup J_2 \cup J_3) & \text {if $x=0$}\\
                    \{x\} \times (J_1\cup J_2) & \text {if $x\in C_n$ for $n \equiv 1 \mod 3$}, \\
                    \{x\} \times (J_2\cup J_3) & \text {if $x\in C_n$ for $n \equiv 2 \mod 3$}, \\
                    \{x\} \times (J_3\cup J_1) & \text {if $x\in C_n$ for $n \equiv 0 \mod 3$}, \\
                    \emptyset & \text {otherwise}~,
      \end{cases}
$$
see Fig.2. Then $M$ is compact and $\{0\} \times \{c\} \in \Lint_s(M)$. In fact all the points
of $M$ except of the end-points of the stars $M_x$, $x\in p(M)$, belong to $\Lint_s(M)$.
\hfill $\square$

 \begin{figure}[ht]
  \label{Fig:2}
  \begin{center}
   \includegraphics[width=8.0cm]{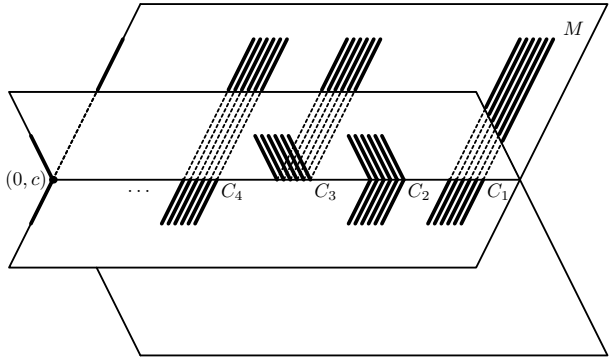}
   \caption{$(0; c)$ is a strongly star-like interior point of M.}
  \end{center}
 \end{figure}
\end{example}

Notice  that $\Lint_s (M)$ is open in $M$ (but not necessarily in $E$)
and $\End_s(M)$ is closed in $M$ (hence closed in $E$). By comparing Definitions~\ref{D:star-like in bundle} and \ref{D:strongly star-like} observe that
\begin{equation}\label{E:strong}
{\Lint}_s (M) \subseteq \Lint (M) = \bigcup _{b\in B} \Lint (M_b), \quad {\End} _s (M) \supseteq \End (M) = \bigcup_{b\in B}
\End (M_b).
\end{equation}
In general neither of these two inclusions is an equality. For
$M\subseteq E$ and $b\in B$ we will further use the notation
$$
M_b^{S_s} = M_b \cap {\Lint}_s (M) = ({\Lint}_s (M))_b~.
$$

\begin{example}\label{Ex:exampleB}
Let $E=B \times \Gamma$ with $B=[-1,1]$ and $\Gamma= [0,3]$. Let $C$
be a Cantor set with $\min C =0$, $\max C =1$ and let $M= ([-1,0]
\times [ 0,1]) \cup (C\times [ 1,2 ]) \cup (\{0\}\times [ 2,3])$.
Then $\Lint_s(M) = ([ -1,0 ] \times (0,1)) \cup (C\times (1,2)) \cup
(\{0\}\times (2,3))$. So, $M_0^{S_s} = \{0\}\times ((0,1)\cup (1,2)
\cup (2,3))$ while $\Lint (M_0) = \{0\}\times (0,3)$. \hfill $\square$
\end{example}

\begin{lemma}\label{L:Ends}
Let $(E, B, p, \Gamma)$ be a compact graph bundle and $M\subseteq
E$ a compact set. Then $${\End} _s M = \overline{\End M}.$$
\end{lemma}

\begin{proof} Without loss of generality we may assume that
$E = B \times \Gamma$. One inclusion is trivial by~(\ref{E:strong}).
To prove the other one, suppose that there is a point $x\in \End_s(M)
\setminus \overline{\End M}$. Then, if the second coordinate of
$x$ has order $m$ in $\Gamma$, we have $m\geq 2$ (otherwise $x$ would be
in $\End (M)$) and some $E$-open neighborhood $O\times \Sigma _m$ of $x$ is
disjoint with $\End (M)$. Hence, if $z\in O$ then the set $(\{z\}\times \Sigma _m)\cap M$
is empty or is of the form $\{z\} \times \Sigma_{k(z),z}$ where $k(z) \in \{2,\dots, m\}$
and $\Sigma_{k(z),z}$ is a full sub-star of $\Sigma_m$ (otherwise it would necessarily
contain a point from $\End (M_z)$). It follows that $x\in \Lint_s(M)$, a contradiction.  \hfill $\square$
\end{proof}

\begin{lemma}\label{L:End}
Let $(E, B, p, \Gamma)$ be a compact graph bundle and $M\subseteq E$ a compact set.
If $ {\End}_s (M) = M$ then $M$ is nowhere dense in $E$.
\end{lemma}

\begin{proof}
If $M$ is somewhere dense in $E$ then, being closed, has nonempty
interior in $E$. It is clear that this interior contains a point
which belongs to ${\Lint}_s (M)$.  \hfill $\square$
\end{proof}

\begin{lemma}\label{L:End2}
Let $(E, B, p, \Gamma)$ be a compact graph bundle and $M\subseteq E$
a compact set with $p(M)=B$. If $\,{\End} (M) = \emptyset$ then $M$
has nonempty interior in $E$.
\end{lemma}

\begin{proof}
We may assume that $E=B\times \Gamma$. Let $K_1, K_2, \dots,
K_k$ be the  list of circles in $\Gamma$. For $i=1,2,\dots,
k$, let $B^{(i)}$ be the set of points $b\in B$ such that $M_b$
contains $\{b\}\times K_i$. The set $M$ is closed and so all the
sets $B^{(i)}$ are closed. Since $p(M)=B$ and ${\End} (M) =
\emptyset$, we have $B=\bigcup _{i=1}^k B^{(i)}$ and since the metric
space $B$ is compact (hence second category), there is $j\in
\{1,2,\dots,k\}$ such that the (closed) set $B^{(j)}$ has nonempty
interior. Since $\Gamma$ is a graph, it follows that $M$ has
nonempty interior in $E$.  \hfill $\square$
\end{proof}

Trivial examples show that the converse statements to the previous two lemmas are not true.

\begin{lemma}\label{L:Lints}
Let $E = B\times \Gamma$ be a compact graph bundle, $M\subseteq E$ a
compact set and $a\in B$. Suppose that $\Delta = \{a\}\times
\Delta^{\Gamma}$ is a compact subset of $M_a^{S_s}$. If $W$ is a
sufficiently small open neighborhood of $a$ and $U$ is a
sufficiently small open neighborhood of $\Delta^\Gamma$ then the
$E$-open neighborhood $W\times U$ of $\Delta$ has the following
properties:
\begin{itemize}
\item[$\bullet$] The corresponding $M$-open neighborhood $\mathcal{D}= M \cap {(W\times U)}$ of $\Delta$
is a subset of $\Lint_s (M)$.
\item[$\bullet$] If we write $\mathcal{D}_z = \{z\} \times \mathcal{D}_z^{\Gamma}$,
then $\mathcal{D}_z^{\Gamma} \subseteq \mathcal{D}_a^{\Gamma}$ and
$\overline{\mathcal{D}_z^{\Gamma}} \setminus \mathcal{D}_z^{\Gamma}
\subseteq \overline{\mathcal{D}_a^{\Gamma}} \setminus
\mathcal{D}_a^{\Gamma}$ whenever $z\in p(\mathcal{D})$.
\item[$\bullet$] The set $p(\mathcal D)$ is closed in $W$, hence it is a Baire space.
\end{itemize}
\end{lemma}

\begin{proof} Since $\Delta$ is compact, it can be covered by a finite family
of $M$-open sets $\mathcal{G}^j= M \cap {(O_j\times
\Sigma_{N(j)})}$, $j=1,\dots,r$, where $\mathcal{G}^j$ are some
canonical $\Lint_s(M)$-neighborhoods of points in $\Delta$. Put $W =
\bigcap_{j=1}^r O_j$ and $U= \bigcup_{j=1}^r \Sigma_{N(j)}$. We
prove that $\mathcal{D}= M \cap {(W\times U)}$ satisfies all the
requirements. First, it is obvious that $\mathcal D$ is an $M$-open
neighborhood of $\Delta$ and $\mathcal D \subseteq \Lint_s(M)$.
Further notice that if we denote, for $j=1,\dots,r$,
$$
M^j = M \cap (W\times \Sigma_{N(j)})
$$
then $a\in p(M^j)$, $(M^j)_a = \{a\}\times \Sigma^j$ where
$\Sigma^j$ is a full sub-star of $\Sigma_{N(j)}$ and, for $z\in
p(M^j)$, $(M^j)_z = \{z\}\times \Sigma_{j(z),z}$ where
$\Sigma_{j(z),z}$ is a full sub-star of $\Sigma^j =
\Sigma_{j(a),a}$. Thus
$$
\mathcal{D}= M \cap {(W\times U)} = \bigcup_{j=1}^r M^j =
\bigcup_{j=1}^r \bigcup_{z\in p(M^j)}(\{z\}\times \Sigma_{j(z),z})~.
$$
Fix $z\in p(\mathcal D) = \bigcup_{j=1}^r p(M^j)$. Since
$$
\mathcal D_z^\Gamma = \bigcup_{j=1}^r \Sigma_{j(z),z}, \quad
\text{in particular} \quad \mathcal D_a^\Gamma = \bigcup_{j=1}^r
\Sigma^j,
$$
we get $\mathcal D_z^\Gamma \subseteq \mathcal D_a^\Gamma$. Hence
$\overline{\mathcal D_z^\Gamma} \subseteq \overline{\mathcal
D_a^\Gamma}$ and so, to prove that
$\overline{\mathcal{D}_z^{\Gamma}} \setminus \mathcal{D}_z^{\Gamma}
\subseteq \overline{\mathcal{D}_a^{\Gamma}} \setminus
\mathcal{D}_a^{\Gamma}$, it is sufficient to show that the
assumption that some point $q \in \overline{\mathcal{D}_z^{\Gamma}}
\setminus \mathcal{D}_z^{\Gamma}$ belongs to $\mathcal D_a^\Gamma$,
leads to a contradiction. To this end consider such a point $q$.
Since $q \in \mathcal D_a^\Gamma$, there is $j \in \{1,\dots,r\}$
such that $q\in \Sigma^j$ and so $q\in U$. On the other hand, $q\in
\overline{\mathcal D_z^\Gamma}$ and so $(z,q)\in \overline{\mathcal
D_z} \subseteq \overline{M} = M$. Also, $(z,q)\in W\times U$ because
$z\in p(\mathcal D) \subseteq W$ and $q\in U$. Thus, $(z,q)\in
M\cap(W\times U) = \mathcal D$ which implies that $q\in \mathcal
D_z^\Gamma$, a~contradiction.

Now we prove that $p(\mathcal D)$ is closed in $W$. It can
 be seen from the definition that if $\mathcal{G}$ is a
canonical $\Lint_s(M)$-neighborhood of a point $x\in M \subseteq
B\times \Gamma$ and for each $z\in p(\mathcal G)$ we put $\mathcal G
_z = \{z\}\times \mathcal G_z^\Gamma$, then the family $\{\mathcal
G_z^\Gamma :\, z\in p(\mathcal G)\}$ is finite. Since $\mathcal D$
was defined using only finitely many such canonical
$\Lint_s(M)$-neighborhoods, we get that also the family $\{\mathcal
D_z^\Gamma :\, z\in p(\mathcal D)\}$ is finite. Therefore, if
$p(\mathcal D) \ni z_n \to z \in W$, we may (passing to a
subsequence if necessary) assume that all sets $\mathcal
D_{z_n}^\Gamma$ are the same. But then, since $M$ is closed,
obviously $\mathcal D_{z}$ is nonempty and so $z\in p(\mathcal D)$.

So, the set $p(\mathcal D)$ is closed (hence is of type
$G_{\delta}$) in the metric space $W$. Since $W$ is open in $B$,
this implies that $p(\mathcal D)$ is $G_{\delta}$ in the compact
space $B$. Thus $p(\mathcal D)$ is a topologically
complete (i.e. completely metrizable) space, hence a Baire space
(see, e.g.,~\cite[Theorems 12.1 and 9.1]{Oxt}).   \hfill $\square$
\end{proof}

In the situation from Lemma~\ref{L:Lints}, let $\Delta \subseteq
M_a^{S_s}$ be \emph{connected}. Then it is a graph and
obviously there exist $m,n\geq 0$ such that every sufficiently small
connected open $\Gamma_a$-neighborhood $V$ of $\Delta$ has the
following properties:
\begin{itemize}
\item[$\bullet$] $V$ is connected and (see Lemma~\ref{L:Lints}) $M\cap V \subseteq M_a^{S_s}$,
\item[$\bullet$] $V\setminus \Delta$ consists of pairwise disjoint open arcs $\{a\}\times I_1^{\Gamma},\dots, \{a\}\times I_m^{\Gamma},
\{a\}\times J_1^{\Gamma},\dots,$ \ $\{a\}\times J_n^{\Gamma}$ where the
arcs $\{a\}\times I_i^{\Gamma}$ are subsets of $M_a^{S_s}$ and the
arcs $\{a\}\times J_i^{\Gamma}$ are disjoint from $M_a$. Each of
these arcs is attached to $\Delta$ at an end-point of $\Delta$ or at
a ramification point of $\Gamma_a$ (an end-point of $\Delta$ can
simultaneously be a ramification point of $\Gamma_a$).
\end{itemize}

We extend the notion of a ramification point as follows. If $G$ is a (not
necessarily closed and not necessarily connected) subset of a graph $\Gamma$
and $g\in G$, we say that $g$ is a \emph{ramification point of $G$}
if there is a $G$-open neighborhood of $g$ which has the form of
an open $r$-star with $r\geq 3$ and with central point $g$.

By an \emph{open graph} we mean a graph without its end-points if it
has any.  So, since a graph is a union of finitely many connected
graphs, an open graph is a union of finitely many connected open
graphs, whose closures are pairwise disjoint. Notice that, by this
definition, a graph having no end-points (in particular, a circle) is
also an open graph and that a circle with one point removed is not
an open graph. If an open graph $G$ is a subset of a graph $\Gamma$
then $G$ need not be an open set in $\Gamma$. Each ramification
point of $G$ is a ramification point of $\Gamma$ but the converse is
not true in general. If $\Gamma$ is a graph and $G\subseteq \Gamma$
is an open graph, by the end-points of $G$ we mean the end-points of
the (closed) graph $\overline{G}$. It follows from the definition of
strongly star-like interior points that the set $M_a^{S_s}$ is open
in the topology of $M_a$ (though not necessarily open in the
topology of $\Gamma_a$). Its connected components are not
necessarily open graphs. For instance, $M_a^{S_s}$ can be a circle
with one point removed. In any case, $M_a^{S_s}$ is a subset of
$\Gamma_a$ and so the notion of a ramification point can be applied
to it.

In the following two technical lemmas we keep the notation from Lemma~\ref{L:Lints}.

\begin{lemma}\label{L:LintsconnectedA}
Let $E = B\times \Gamma$ be a compact graph bundle, $M\subseteq E$ a
compact set and $a\in B$. Suppose that $\Delta = \{a\}\times
\Delta^{\Gamma}\subseteq M_a^{S_s}$ is an arc or a circle and does not contain any ramification point of
$M_a^{S_s}$. Then for any sufficiently small open neighborhood $W$
of $a$ and any sufficiently small connected open neighborhood $U$ of
$\Delta^\Gamma$ as in Lemma~\ref{L:Lints},  it holds that $\mathcal D = W^* \times U^*$, i.e. $\mathcal D$
has the structure of a direct product. Here $a\in W^* \subseteq W$
is some not necessarily $B$-open set. If $\Delta$ is a circle then
$\{a\}\times U^*$ coincides with $\Delta$ and if $\Delta$ is an arc
then $\{a\}\times U^*$ is an open arc containing $\Delta$ (and still
containing no ramification point of $M_a^{S_s}$).
\end{lemma}

\begin{proof}
Every point from $\Delta$ has an
$M_a$-neighborhood in the form of an open arc and so, since $\Delta \subseteq \Lint_s (M)$, $\Delta$ can be covered by a finite
family of canonical $\Lint_s(M)$-neighborhoods of points from
$\Delta$ which have the form (see the proof of Lemma~\ref{L:Lints})
$$
\mathcal{G}^j= M \cap {(O_j\times \Sigma_{N(j)})} = V^j \times
\Sigma_2^j~.
$$
Here $V^j$ is a (not necessarily $B$-open) set containing $a$ and
$\Sigma_2^j$ is an open arc in $\Gamma$ such that $\{a\}\times
\Sigma_2^j \subseteq \Lint_s(M)$ contains no ramification point of
$M_a^{S_s}$.

If two open arcs $\Sigma_2^j$ and $\Sigma_2^i$ intersect and $z\in
\bigcap O_j$ then $z\in V^j$ if and only if $z\in V^i$. This
together with the fact that $\Delta$ is connected gives that if
$z\in \bigcap O_j$ then $z$ belongs to all of the sets $V^j$
whenever it belongs to one of them. Now let $U$ be any sufficiently
small connected open neighborhood of $\Delta^\Gamma$ so that
$(\{a\}\times U) \cap M_a \subseteq \{a\}\times \bigcup \Sigma_2^j$.
Further, let $W\subseteq \bigcap O_j$ be any open neighborhood of
$a$. Then the claim holds with $U^* = U\cap \bigcup \Sigma_2^j$ and
$
W^* = \{z\in W :\, z \in V^j \text { for some (hence for all) }
j\}~.
$  \hfill $\square$
\end{proof}

\begin{example}\label{Ex:examplepredC}
Let us  return to  Example~\ref{Ex:exampleA}.
Denoting by $\Delta$ an arc in $M_0$ containing the ramification
point $c$, we see that without assuming that $\Delta$ contains no
ramification point of $M_a^{S_s}$, in
Lemma~\ref{L:LintsconnectedA} one cannot ensure the existence of
$\mathcal{D}$ in the form of a direct product. Further, if $\Delta$
does not contain the ramification point $c$ and is a sub-arc of, say,
$J_1$ we see that one cannot claim  that $W^*$ exists in the
class of $B$-open sets. \hfill $\square$
\end{example}

\begin{example}\label{Ex:exampleC}
Let us  return to  Example~\ref{Ex:exampleB} and put
$\Delta = \{0\} \times \{1/2, 3/2, 5/2\}$. Then $\Delta \subseteq
M_0^{S_s}$ and it does not contain any ramification point of
$M_0^{S_s}$ (even any ramification point of $\Gamma_0$). However,
$\Delta$ is disconnected and there is no $M$-open neighborhood of
$\Delta$ of the product form $W^* \times U^*$. \hfill $\square$
\end{example}

\begin{lemma}\label{L:LintsconnectedB}Let $E = B\times \Gamma$ be a compact graph bundle, $M\subseteq E$ a
compact set and $a\in B$. Suppose that $\Delta = \{a\}\times
\Delta^{\Gamma}\subseteq M_a^{S_s}$ is  a graph possibly degenerate to a singleton (and possibly containing ramification
points of $M_a^{S_s}$, which may or may not be ramification points
of $\Delta$). Then for any sufficiently small open neighborhood $W$
of $a$ and any sufficiently small connected open neighborhood $U$ of
$\Delta^\Gamma$ as in Lemma~\ref{L:Lints}, 
 the following holds.
\begin{itemize}
\item[$\bullet$] $\left ( \{a\} \times (U \setminus \Delta^{\Gamma}  ) \right ) \cap M \subseteq M_a^{S_s}$
is empty or consists of  pairwise disjoint open arcs $\{a\}\times
I_1^{\Gamma},\dots, \{a\}\times I_m^{\Gamma}$ ($m\geq 0$ being
finite and independent on $U$, since $U$ is small enough; $m=0$
means that the described set is empty).
\item[$\bullet$] For each $i=1,\dots, m$ the open arc $\{a\}\times I_i^{\Gamma}$  is attached to $\Delta$ at a point
$p_i=(a,p_i^\Gamma)$ which is an end-point of $\Delta$ or a
ramification point of $M_a^{S_s}$ (an end-point of $\Delta$ can
simultaneously be a ramification point of $M_a^{S_s}$ and it can be
$p_i=p_j$ even if $i\neq j$), and at each of the end-points of
$\Delta$ there is at least one such open arc attached to it. Here
for every $i$, the closure of $\{a\}\times I_i^{\Gamma}$ is an arc
and any two of the sets $\Delta, \overline{\{a\}\times I_i^\Gamma}$,
$i=1,\dots,m$ are either disjoint or intersect only at one of the
`attaching' points $p_i$.
\item[$\bullet$] $\mathcal{D}_a = \Delta$ or $\mathcal{D}_a = \Delta \cup \bigcup _{i=1}^m (\{a\}\times I_i^\Gamma)$,
depending on whether $m=0$ or $m\geq 1$. So, $\mathcal{D}_a$ is an open graph.
\item[$\bullet$] The structure of the corresponding $M$-open neighborhood
$\mathcal{D}= M \cap {(W\times U)}\subseteq \Lint_s(M)$ of $\Delta$
is such that for any $z\in p(\mathcal{D})$, $\mathcal{D}_z^{\Gamma}$
is a union of finitely many open graphs whose closures are pairwise
disjoint, $\mathcal{D}_z^{\Gamma} \subseteq \mathcal{D}_a^{\Gamma}$
and $\End(\overline{\mathcal{D}_z^{\Gamma}}) \subseteq
\End(\overline{\mathcal{D}_a^{\Gamma}})$.
\item[$\bullet$] For any $z\in p(\mathcal{D})$, each of the connected components of $\mathcal{D}_z$ is the union of a
(nonempty, closed) possibly degenerate subgraph of $\{z\}\times \Delta^{\Gamma}$ and some (possibly zero) of the open arcs
$\{z\}\times I_i^{\Gamma}$ with the `attaching' points $(z,p_i^\Gamma)$ belonging to $\mathcal{D}_z$.
If this subgraph is nondegenerate and does have one or more end-points, then at
each of these end-points there is at least \emph{one} of these open arcs attached to it.
If the subgraph is a singleton (which may happen even if $\Delta$ is nondegenerate) then at least \emph{two}
of these open arcs are attached to it.
\end{itemize}
In particular, if $\Delta$ is a tree, possibly degenerate to a
singleton, then:
\begin{itemize}
\item[$\bullet$] For each $z\in p(\mathcal{D})$, the set $\mathcal{D}_z$ contains (a nonempty closed subgraph of $\{z\}\times \Delta^{\Gamma}$,
possibly disconnected, possibly degenerate to a finite set, and) at least \emph{two} of the open arcs $\{z\}\times I_i^{\Gamma}$,
with the `attaching' points $(z,p_i^\Gamma)$ belonging to $\mathcal{D}_z$.
\end{itemize}

\end{lemma}

Of course, if $\Delta$ is a singleton, then the last statement of the lemma does not say anything more than the definition
of a strongly star-like interior point of $M$.

\begin{proof} The arguments are completely analogous to those used in  the proof of Lemma Lemma~\ref{L:LintsconnectedA}. In fact, 
 the first three parts are just consequences of our definitions
of $M_a^{S_s}$, ramification points, endpoints and open  graphs.
The rest follows from Lemmas~\ref{L:Lints}, \ref{L:LintsconnectedA} and the remarks above
Lemma~\ref{L:LintsconnectedA}. (Note that a key role is played by the
fact that $\mathcal{D}\subseteq \Lint_s(M)$. For instance, if the intersection of $\mathcal{D}_z$ with
$\{z\}\times \Delta^{\Gamma}$ is a singleton, then at least two open
arcs have to be attached to this singleton, otherwise
$\mathcal{D}_z$ could not be a subset of $\Lint_s(M)$.)  \hfill $\square$
\end{proof}


\section{Proof of Theorem A}\label{S:proofA}


We will use the notation $F_z = F|_{\Gamma_z}$. So, $F_z$ is a map from $\Gamma_z$ into $\Gamma_{f(z)}$.

We start with the following result  partially describing $F$ on its minimal sets in case (A2) of  our Theorem~A.  Its
use simplifies arguments in the proof of Theorem A.

\begin{proposition}\label{P:monot}
Let the assumptions of Theorem~A be satisfied. Let $I_a$ be a closed
arc and $T_b$ be a tree such that $I_a \subseteq M_a^{S_s}$, $T_b
\subseteq M_b$ and  $F(I_a) \subseteq T_b$. If the interior of $I_a$
does not contain any ramification point of $M_a^{S_s}$ then
$F|_{I_a}$ is monotone (hence $F(I_a)$ is an arc or a point).
\end{proposition}

The statement in the parentheses is obvious since a monotone image of an arc cannot be a nondegenerate tree.
Both cases (i.e., $F(I_a)$ is an arc or a point) occur in the example of a noninvertible fibre-preserving
minimal map on the torus in~\cite{KST} (the base is a `horizontal' circle, the fibres are `vertical' circles).
Since in this example there is a vertical arc mapped by $F$ into a point while the vertical circle containing
this arc is mapped onto a circle, the example also shows that the proposition would not be true if $T_b$ were
allowed to contain a circle.

\begin{proof}
It is sufficient to prove a weaker version of the proposition which
is obtained by adding the assumption that neither the end-points of
$I_a$ are ramification points of $M_a^{S_s}$. For if one or both end
points of $I_a$ are ramification points of $M_a^{S_s}$ then, by
applying such a weaker proposition to all sub-arcs $J_a$ of $I_a$
which do not contain end-points of $I_a$, we get the monotonicity of
$F$ on the whole interior of $I_a$. Since the $F$-image of this
interior is a point or a (not necessarily closed) arc and $T_b$ does
not contain a circle, $F$ is obviously monotone on $I_a$.

So, let $I_a$ contain no ramification point of $M_a^{S_s}$ and
suppose, on the contrary, that $F|_{I_a}$ is not monotone. Then
there exists $q \in T_b$ such that $(F|_{I_a})^{-1}(q)\subseteq I_a$
is not connected. Take two points $u,v$ in two different connected
components of $(F|_{I_a})^{-1}(q)$ and consider the (unique) arc
$J_a \subseteq I_a$ with the end-points $u,v$. From the choice of
$u,v$ it follows that there is a point $w \in J_a$ with $F(w)\neq
q$. This point $w$ partitions $J_a$ into two nondegenerate closed
sub-arcs $J_a^1$ and $J_a^2$. The set $F(J_a) = F_a(J_a) \subseteq
T_b$ is a nontrivial continuum (hence a tree) and each of the sets
$F(J_a^1)$ and $F(J_a^2)$ contains the (unique) arc in $T_b$ having
the end-points $F(w)$ and $q$. It follows that the arc $J_a$ contains
two disjoint closed nondegenerate sub-arcs $T_a^1, T_a^2$ such that
$F(T_a^1)$ and $F(T_a^2)$ are closed arcs with $F(T_a^1) \subseteq
\Inte F(T_a^2)$ (where by $\Inte F(T_a^2)$ we mean the arc
$F(T_a^2)$ without its end-points).

Now, since we will work only with some neighborhood of $a$,  we may assume that $E$ has
the structure of a product space, i.e. $E=B\times \Gamma$. So $I_a$ has the form $\{a\} \times I$ and similarly
$T_a^1=\{a\}\times T^1$ and $T_a^2=\{a\}\times T^2$. By Lemma~\ref{L:LintsconnectedA}, there is an $M$-open neighborhood
$\mathcal D$ of $I_a$ which has the product form $\mathcal D = W^* \times U^*$ for some (not necessarily $B$-open)
set $W^* \ni a$ and some open arc $U^*$ containing $I$.

Since $F_a(\{a\}\times T^1) \subseteq \Inte F_a(\{a\}\times T^2)$ and since (by replacing $T^1$ by a smaller arc
if necessary) we may assume that the arc $F_a(\{a\}\times T^1)$ does not contain any ramification point of $\Gamma_b$,
we have $F_x(\{x\}\times T^1) \subseteq \Inte F_x(\{x\}\times T^2)$ also for all $x$ sufficiently close to $a$.
By replacing $W^*$ by its intersection with a small open neighborhood of $a$ if necessary, we may assume that this
is the case for all $x\in W^*$. Then
$$
F|_M (W^*\times \Inte T^1) \subseteq F|_M (W^*\times T^2)\subseteq  F|_M (M\setminus (W^*\times \Inte T^1))~.
$$
Hence the nonempty $M$-open set $W^* \times \Inte T^1$ is redundant for $F|_M$ which contradicts the minimality of $F|_M$.
\end{proof}

When $M\subseteq E$ and $\beta \in \End (M)$, i.e. $\beta \in \End (M_b)$ where $b=p(\beta)$, then still it can
happen that there is an open arc $J\subseteq M_b$ such that $\beta \in J$ (e.g., let $\Gamma_b$ be a 3-star $S_3$
with central point $\beta$, $M_b$ be the union of a 2-star $S_2$ with the same central point $\beta$ and a sequence
of points lying in $S_3\setminus S_2$ and converging to $\beta$). However,  the following lemma holds.

\begin{lemma}\label{L:zmena bety}
Let the assumptions of Theorem~A be satisfied. Suppose that there exists a point in $\End(M) \setminus F(\End_s(M))$.
Then in the same fibre there exists also a point $\beta \in \End(M) \setminus F(\End_s(M))$ such that no open arc
containing $\beta$ exists in $M_b$, $b=p(\beta)$.
\end{lemma}

\begin{proof}
Choose any $\beta ' \in \End(M) \setminus F(\End_s(M))$ and denote $p(\beta ') =b$. Suppose that $\beta '$ is
contained in an open arc $J\subseteq M_b$. Then, since $\beta ' \notin \Lint (M_b)$, the point $\beta '$ is
necessarily a ramification point of $\Gamma_b$ and in one of the small open branches emanating from $\beta '$ there
are both a sequence of points in $M_b$ converging to $\beta '$ and a sequence of points in $\Gamma_b \setminus M_b$
converging to $\beta '$. Then this branch obviously contains also a sequence of points $\beta _n \to \beta '$ such that,
for every $n$, $\beta_n \in \End(M_b)$ and no open arc in $M_b$ contains $\beta_n$. Now it is sufficient to put
$\beta = \beta_n$ for a sufficiently large $n$, because $F(\End_s(M))$ is a closed set which does not contain $\beta'$.
\end{proof}

We are finally ready to prove our Theorem A.

\vspace{2mm}

\noindent {\bf Theorem A.} {\it
Let $M$ be a minimal set (with full projection onto the base) of a fibre-preserving map in a compact graph bundle $(E,B,p,\Gamma)$. Then there are two mutually exclusive possibilities:
\begin{enumerate}
\item [{ (A1)}] $\overline {\End (M)}=M$ (and this holds if and only if $M$ is nowhere dense in $E$);
\item [{ (A2)}] $\End(M) = \emptyset$ (and this holds if and only if $M$ has nonempty interior in $E$). 
\end{enumerate}
In particular, the fibre-preserving maps in tree bundles have only nowhere dense minimal sets.}

\vspace{2mm}

\begin{proof}
Also the last claim is obvious, since if
$\Gamma$ is a tree then $\End(M) \neq \emptyset$ and we are
therefore in the case (A1). Thus, taking into account
Lemmas~\ref{L:Ends}, \ref{L:End} and \ref{L:End2}, it remains to
prove the dichotomy: either $\overline {\End (M)}=M$ or $\End(M) =
\emptyset$. To this end suppose that $\End M \neq \emptyset$. To
prove that then $\overline {\End (M)}=M$, it sufficies to show
that every point in $\End(M)$ has an $F$-pre-image in $\End_s (M)$.
Indeed, suppose for a moment that 
$F(\End_s (M)) \supseteq \End (M)$. Then $F(\End _s(M)) \supseteq
\overline{\End(M)}= \End_s (M)$ by Lemma~\ref{L:Ends}. It follows
that the nonempty  closed set $\End_s(M)$ is not a proper subset
of $M$, otherwise $(M, F|_M)$ is not minimal, see the equivalence 
(1) $\Leftrightarrow$ (3) in Subsection~\ref{SS:minimality}. So, $\End_s(M) = M$ whence by
Lemma~\ref{L:Ends} we get $\overline {\End (M)}=M$.

Thus, to finish the proof, we suppose that there is a point
$\beta \in \End (M) \setminus F(\End_s(M))$ and we want to get a
contradiction. If we denote $p(\beta) = b$, by Lemma~\ref{L:zmena
bety} we can assume that
\begin{equation}\label{E:beta}
\text{there is no open arc in $M_b$ containing $\beta$.}
\end{equation}

Since $F(M)=M$ and $\beta \notin F(\End_s(M))$, there is a point
$\alpha \in \Lint_s (M)$ with $F(\alpha) =\beta$. Denote $p(\alpha)
=a$. From now on we will work only with neighborhoods of $\Gamma_a$
and $\Gamma_b$ and so, due to the local triviality of the graph
bundle, we may assume that $E= B\times \Gamma$. Let $\ord (\beta,
\Gamma_b) = r \geq 1$, i.e. $\beta = (b, \beta^\Gamma)$ where
$\beta^\Gamma$ is the central point of an open $r$-star in $\Gamma$.
Since the set $F(\End_s (M))$ is closed in $E$ and does not contain
$\beta$, for some $B$-open neighborhood $O$ of $b$ and some open
$r$-star $\Sigma_r$ with the central point $\beta^\Gamma$ the open
$E$-neighborhood $\mathcal{O}^* = O \times \Sigma_r$ and hence also
the $M$-open neighborhood $\mathcal{O}= \mathcal{O}^*\cap M$ of
$\beta$ are disjoint from $F(\End_s(M))$. In view of~(\ref{E:beta}),

\begin{equation}\label{E:beta2}
\begin{split}
\text{the connected component of $M_b \cap \mathcal O$ containing $\beta$ is either the singleton $\beta$} \\
\text{or a (half-closed or closed) arc whose one end-point is
$\beta$.}
\end{split}
\end{equation}

Recall that $F_z = F|_{\Gamma_z}$. Consider the map $F_a:
\Gamma_a \to \Gamma_{b}$ and choose that connected component
$\Delta$ of the set $F_a^{-1}(\beta)\cap M$ which contains the point
$\alpha$. Since $\beta \notin F(\End_s(M))$, we have $\Delta
\subseteq \Lint_s (M)$. The set $\Delta$ is closed, so it is the
singleton $\alpha$ or a (nondegenerate closed) connected subgraph of
$\Gamma_a$ containing $\alpha$. Let $\Delta^{\Gamma}$ be the
counterpart of $\Delta$ in $\Gamma$, i.e., $\Delta=\{a\}\times
\Delta^{\Gamma}$.

Let $W$ be a $B$-open neighborhood of $a$ and $U$ be a connected
$\Gamma$-open neighborhood of $\Delta^{\Gamma}$, both as small as
Lemma~\ref{L:LintsconnectedB} requires.  In what follows,
$\mathcal{D}= M \cap {(W\times U)}\subseteq \Lint_s(M)$,
$I_i^{\Gamma}$ and $p_i = (a, p_i^\Gamma)$ will have the meaning
from this lemma. We will also consider the half-closed arcs
$A_i^\Gamma = \{p_i^\Gamma\} \cup I_i^\Gamma$, $i=1,\dots, m$. Since
$F(\Delta)$ is just the singleton $\beta$, we may also assume that
$W$ and $U$ are small enough to give
\begin{equation}\label{E:nocircle}
\text{$F(\mathcal{D}) \subseteq \mathcal O$, hence none of the sets
$F(\mathcal D_z)$, $z\in W$, contains a circle.}
\end{equation}

\medskip

\emph{Claim.}  There is $d\in p(\mathcal{D})$ such that
$\mathcal{D}_d^\Gamma$ contains no circle (and each component
of $\mathcal{D}_d^\Gamma$ is nondegenerate since
$\mathcal D \subseteq \Lint_s(M)$ and $\mathcal D$ is $M$-open).
Moreover, $m\geq 2$ and $\mathcal{D}_d$ contains at least two
different half-closed arcs from the list $\{d\}\times A_i^{\Gamma}$,
$i=1,\dots, m$.

\medskip

\emph{Prof of Claim.} Let $C_1^\Gamma,  \dots, C_q^\Gamma$, $q\geq
0$, be the list of all (not necessarily pairwise disjoint) circles
in $\Delta^\Gamma$. If $z\in p(\mathcal D)$ then, by
Lemma~\ref{L:LintsconnectedB}, $\mathcal D_z^\Gamma \subseteq
\mathcal D_a^\Gamma = \Delta^{\Gamma} \cup \bigcup _{i=1}^m I_i^\Gamma$ and
$\mathcal D_a^\Gamma$ contains only those circles which are
contained in $\Delta^\Gamma$. So, if $D_z^\Gamma$ contains a circle,
it is necessarily a circle from the list $C_1^\Gamma,  \dots,
C_q^\Gamma$. Denote
\begin{equation}\label{E:n1}
K_i = \{z\in p(\mathcal D):\, \mathcal D_z^\Gamma \supseteq
C_i^\Gamma\}, \quad i=1,\dots,q~. \nonumber
\end{equation}
To prove the claim suppose, on the contrary, that for every $z\in
p(\mathcal D)$, $\mathcal D_z^\Gamma$ contains a circle. Then $q\geq
1$ and
\begin{equation}\label{E:n2}
p(\mathcal{D}) =\bigcup_{i=1}^q K_i~. \nonumber
\end{equation}
Each of the sets $K_i$, $i=1,\dots,q$, is obviously closed in the
set $p(\mathcal D)$ which is, by Lemma~\ref{L:Lints}, a Baire space.
Hence there is $s\in \{1,\dots,q\}$ with
\begin{equation}\label{E:n3}
\Inte{_{p(\mathcal D)}} K_s \neq \emptyset.
\end{equation}
Now fix an arbitrary $j\in \{1,\dots,q\}$ and an open arc
$L_j^\Gamma$ in $C_j^\Gamma$ such that the closure of $L_j^\Gamma$
contains only points of order $2$ in $\Gamma$ (in particular,
$L_j^\Gamma$ has positive distance from the set $\{p_i^\Gamma:\,
i=1,\dots,m\}$). Observe that then for every $z\in K_j$ the map
$F_z$ is, by Proposition~\ref{P:monot} (see
also~(\ref{E:nocircle})), monotone on $\{z\}\times L_j^\Gamma$ and
so $F_z(\{z\}\times L_j^\Gamma)$ is an open, closed or half-closed
arc, possibly degenerate to a point. Since $F_z(\mathcal D_z)$ is
by~(\ref{E:nocircle}) a tree (which is a uniquely arcwise connected
space), we have that $F_z(\{z\}\times (C_j^\Gamma \setminus
L_j^\Gamma)) \supseteq F_z(\{z\}\times L_j^\Gamma)$. Hence
\begin{equation}\label{E:n4}
F(S \times L_j^\Gamma) \subseteq F(M\setminus (S \times L_j^\Gamma))
\quad \text{for any set $S\subseteq K_j$ , $j\in \{1,\dots,q\}$}~.
\end{equation}
Note also that here $S \times L_j^\Gamma \subseteq M$.

Then by~(\ref{E:n4}), for $j=s$ and $S= \Inte_{p(\mathcal D)} K_s$
we obtain $F(\Inte_{p(\mathcal D)} K_s \times L_s^\Gamma) \subseteq
F(M\setminus (\Inte_{p(\mathcal D)}K_s \times L_s^\Gamma))$.
Therefore, since the set $\emptyset \neq \Inte_{p(\mathcal D)} K_s
\times L_s^\Gamma \subseteq M$ is obviously open in the topology of
$M$, the set $\Inte_{p(\mathcal D)} K_s \times L_s^\Gamma$ is a
redundant open set for $F|_M$, which contradicts the minimality of
$F|_M$. We have thus proved that there exists $d \in p(\mathcal D)$
such that $\mathcal D_d^\Gamma$ contains no circle.

Applying now the last assertion of Lemma \ref{L:LintsconnectedB}, we
find  that $\mathcal{D}_d$ contains at least two different
half-closed arcs from the list $\{d\}\times A_i^{\Gamma}$,
$i=1,\dots, m$. Thus $m \geq 2$ which finishes the proof of the claim.~
\hfill $\checkmark$$\checkmark$$\checkmark$

\medskip

Next, we will replace $W$ by a smaller open neighborhood of $a$ and
$U$ by a smaller connected open neighborhood of $\Delta^\Gamma$ so
that $\mathcal D$ have an additional nice property. We are going to
show how to do that. Note also that the Claim will still work.

Recall that, by the Claim, $m\geq 2$. The attaching points $p_i = (a,
p_i^\Gamma)$, $i=1,2\dots, m$ belong to $\Delta$ and so are mapped
to the point $\beta$. On the other hand, $\Delta$ is disjoint with
the open arcs $\{a\}\times I_i^{\Gamma}$. Therefore each of the sets
$F(\{a\}\times A_i^\Gamma)$ is a \emph{nondegenerate} connected set
in $M_b$ containing $\beta$. Taking into account~(\ref{E:beta2}), we
see that each of these sets is in fact a closed or half-closed arc
containing $\beta$ as one of its end-points  (so that  the connected component of $M_b \cap \mathcal O$
containing $\beta$ is not  a singleton, 
see~(\ref{E:beta2})), and $F(\{a\}\times A_i^\Gamma) \subseteq
F(\{a\}\times A_j^\Gamma)$ or $F(\{a\}\times A_j^\Gamma) \subseteq
F(\{a\}\times A_i^\Gamma)$ whenever $i, j \in \{1,\dots, m\}$. By
replacing the half-closed arcs $A_i^\Gamma$ by shorter ones (i.e.,
by replacing $U$ by a smaller connected open neighborhood of
$\Delta^\Gamma$) if necessary, we may assume that each of the
half-closed arcs $\{a\}\times A_i^\Gamma$ is monotonically
(see~(\ref{E:nocircle}) and Proposition~\ref{P:monot}) mapped by $F$
onto the \emph{same} half closed arc $H$ with the end-point $\beta
\in F(\{a\}\times A_i^\Gamma)$ and another end-point $\beta^* \notin
F(\{a\}\times A_i^\Gamma)$.

Now fix $k\in \{1,\dots,m\}$ and choose a small open arc $J_k =
\{a\}\times J_k^\Gamma$ such that the closure of $J_k$ lies in the
interior of $\{a\}\times A_k^{\Gamma}$ and the closure of $F(J_k)$
lies in the interior of $H$. Then
the closure of $F(\{a\}\times J_k^\Gamma)$ lies in the
interior of $F(\{a\}\times A_i^\Gamma)$ for every $i=1, 2\dots, m$.
By continuity, and replacing $W$ by a smaller neighborhood of $a$ if
necessary, we may assume that
\begin{equation}\label{E:sn1}
\text{$F(\{z\}\times J_k^\Gamma) \subseteq F(\{z\}\times
A_i^\Gamma)$ for every $z\in W$ and $i=1,2\dots,m$.}
\end{equation}
Note that this holds (i.e., such a $J_k^\Gamma$ exists) for any
$k\in \{1,\dots, m\}$.

Now we can finish the proof.  By the Claim, there
exists $d \in p(\mathcal D)$ such that $\mathcal{D}_d$ does not
contain any circle and contains at least two different half-closed
arcs, say $\{d\}\times A_1^{\Gamma}$ and $\{d\}\times A_2^{\Gamma}$.
Both these properties are shared by all the points $z\in
p(\mathcal{D})$ sufficiently close to the point $d$. Indeed, $M$ is
closed and $\Gamma$ contains only finitely many circles and so, if
$z\in p(\mathcal{D})$ is close to $d$, neither the set
$\mathcal{D}_z$ can contain a circle. But then, using the same
argument as for the point $d$ (see the very end of the proof of
the Claim), the set $\mathcal{D}_z$ also contains at least \emph{two} of
the half-closed arcs $\{z\}\times A_i^{\Gamma}$. It follows that for
any $z\in p(\mathcal D)$ close to $d$ there is at least one $i\neq
1$ such that $\{z\}\times A_i^\Gamma \subseteq M$ and so, regardless
of whether $\{z\}\times J_1^\Gamma \subseteq \{z\}\times A_1^\Gamma$
is a subset of $M$ or is disjoint from $M$, the
condition~(\ref{E:sn1}) applied to $k=1$ gives $F(M_z\setminus
(\{z\}\times J_1^\Gamma)) \supseteq F(M_z\cap (\{z\}\times
J_1^\Gamma))$. Hence, for sufficiently small neighborhood
$W_1\subseteq W$ of $d$  we have $F(M\setminus (W_1\times J_1^\Gamma))
\supseteq F(M\cap (W_1\times J_1^\Gamma))$ and so the nonempty
$M$-open set $M\cap (W_1\times J_1^\Gamma)$ is redundant for $F|_M$,
a contradiction with minimality of $F|_M$. \hfill $\square$
\end{proof}


\section{Proof of Theorem~C}\label{S:proofC}


If $M \subseteq E$ is a closed set with $\End (M)
= \emptyset$, we have $\End (M_b) = \emptyset$ for every $b\in B$
and so every set $M_b$ is a (possibly disconnected) graph without
end-points (this in particular means that for every $b\in B$ the set
$M_b$ contains at least one circle). We will be interested in
whether such a graph $M_b$ has a ramification point or not. Of
course, $M_b$ does not have any ramification point if and only if it
is a union of disjoint circles. Denote
$$\mathcal R_B(M) := \{b \in B: M_b \text{ has a ramification point}\}~,$$
$$\mathcal R_E(M) := \{\gamma \in E:  \gamma \text{ is a ramification point of } M_{p(\gamma)}\}~.$$

\begin{lemma}\label{L:ramif1}
Let $E = B\times \Gamma$ be a compact graph bundle and $M\subseteq E$ a closed set with $\End (M) = \emptyset$.
\begin{enumerate}
\item [(a)] If $U$ is an open ball in $B$ with $U\subseteq \mathcal R_B(M)$ then there are an open ball $V\subseteq U$  and a ramification point $q$ of $\Gamma$ such that $V\times\{q\} \subseteq \mathcal R_E(M)$.
\item [(b)] Let $q$ be a ramification point of $\Gamma$ of order $N$ and $V$ be an open ball in $B$ with $V\times\{q\} \subseteq \mathcal R_E(M)$. Let
    an open star $\Sigma_N \subseteq \Gamma$ with central point $q$ be a $\Gamma$-open neighborhood of $q$ (i.e., $\Sigma_N$ contains no ramification point of $\Gamma$ different from $q$). Then there are a full sub-star $\Sigma_k$ of $\Sigma_N$ with $k\geq 3$ and an open ball $W\subseteq V$ in $B$ such that $(W \times \Sigma_N) \cap M = W\times \Sigma_k$ (hence $W\times \Sigma_k$ is an $M$-open set).
\end{enumerate}
\end{lemma}

\begin{proof}
(a) For each $u\in U$ there is $q_u$ in $\Gamma$ such that $(u,q_u)\in M$ is a ramification point of $M_u$. Since there are only finitely many ramification points in $\Gamma$ and the set $M$ is closed, we get that the same $q$ works for all $u$ in a subset of $U$ with nonempty interior.

(b) For all $v\in V$, $(v,q)$ is a ramification point of $M_v$. The neighborhood $\Sigma_N$ of $q$ is a disjoint union of the point $q$ and $N$ open arcs emanating from $q$. If $v\in V$ and $I$ is one of these open arcs then $\{v\}\times I$ is either a subset of $M_v$ or disjoint with $M_v$ (because the graph $M_v$, possibly disconnected, has no end-points). Let $I_i$, $i=1,\dots,r$ be the list of those of the $N$ open arcs for which $\{v\}\times I_i \subseteq M_v$ for at least one $v\in V$. We say that $v\in V$ has \emph{signature} $\lambda = \{i_1,\dots,i_s\}$ if $M_v$ contains from this list just the open arcs $\{v\}\times I_{i_1},\dots, \{v\}\times I_{i_s}$. So, the signature $\lambda$ is a subset (with cardinality at least three) of
$\{1,\dots,r\}$. Let $\Lambda$ be the family of signatures of all
points $v\in V$. Then $\Lambda$ is finite and if $S_\lambda$
is the set of all points $v\in V$ with signature $\lambda$, then $V=
\bigcup_{\lambda \in \Lambda} S_\lambda$. Until the end of the proof we will work in (the topology of) the Baire space $V$. Denote by $\overline{S_\lambda}$ the closure (in $V$) of $S_\lambda$. We claim that
$\overline{S_\lambda}\setminus S_\lambda$ is closed in
$\overline{S_\lambda}$, i.e., $S_\lambda$ is locally closed (in $V$).
The reason is as follows. If $x\in \overline{S_\lambda}$ has signature $\mu$ then, since $M$ is closed, $\mu \supseteq \lambda$. If $x \in \overline{S_\lambda}\setminus S_\lambda$ then $\mu \supsetneq \lambda$. This property of having the signature strictly larger than $\lambda$ is obviously inherited by the limit of a sequence of points from $\overline{S_\lambda}\setminus S_\lambda$. It follows that $\overline{S_\lambda}\setminus S_\lambda$ is closed. So, applying Lemma~\ref{L:Baire} to the Baire space $V$ we get that there is an open ball $W$ in $V$ (hence $W$ is an open ball in $B$) such that all points $w\in W$ have the same signature $\{i_1,\dots,
i_k\}$ (of cardinality $k\geq 3$). It follows the existence of a full sub-star  $\Sigma_k$ of $\Sigma_N$ with the required properties. \hfill $\square$
\end{proof}

\begin{lemma}\label{L:ramif2} Let $M$ be a minimal set (with full projection onto the base) of a fibre-preserving map $F$ in a direct product graph bundle $E= B \times \Gamma$. Assume that  $\End (M) =\emptyset$.
Suppose that an open ball $V$ in $B$ and a ramification point $q\in \Gamma$ are such that $V\times \{q\} \subseteq \mathcal R_E(M)$. Then there are an open ball $V^*\subseteq V$  and a ramification point $\tilde q$ of $\Gamma$ such that $F(V^*\times \{q\}) = f(V^*) \times \{\tilde q\} \subseteq \mathcal R_E(M)$.
The same is true for closed balls instead of open ones.
\end{lemma}

\begin{proof} Choose $\Sigma_N$, $\Sigma_k$ and $W$ by Lemma~\ref{L:ramif1}(b). It is obviously sufficient to show that $F(W\times \{q\})\subseteq \mathcal R_E(M)$. Indeed, then (since $\Gamma$ has only finitely many ramification points and $F$ is continuous) for any sufficiently small open ball $V^*$ in $B$ such that $V^* \subseteq W$, the second projection of the set $F(V^*\times \{q\}) \subseteq \mathcal R_E(M)$ will be just a singleton $\tilde q$ (a ramification point of $\Gamma$).

So, fix any $a\in W$ (from now on we will write $W_a$ instead of $W$, to indicate that it contains $a$) and put $\alpha = (a, q)$, $\beta = F(\alpha) =
(b,p)$ (of course, $b=f(a)$ and $\beta \in M = \Lint_s(M)$). We are going to prove that $\beta \in \mathcal R_E(M)$.

Suppose, on the contrary, that $\beta \notin \mathcal R_E(M)$. Then $\beta \in \Lint_s
(M)\setminus \mathcal R_E(M)$ and so one can apply Lemma~\ref{L:LintsconnectedA}
to a small arc in $M_b$ containing $\beta$, to obtain that there is an
$M$-open neighborhood of $\beta$ in the form $W_b \times \Sigma_2$
where $W_b$ contains $b$ but it need not be a $B$-neighborhood of $b$ and $p$ is the central point of $\Sigma_2$. Recall that $W_a$ is a $B$-open neighborhood of $a$ and $W_a \times \Sigma_k$ is an $M$-open neighborhood
od $\alpha\in \Lint_s(M)$. Since $F$ is continuous, we may assume that $W_a$ and $\Sigma_k$ are small enough so that $F(W_a \times \Sigma_k)
\subseteq W_b \times \Sigma_2$. We are going to show that there
exists a redundant open set for $F|_M$, which will contradict the
minimality of $F|_M$. To this end consider two cases.

First assume that  there exists $x\in W_a$ such that at least three
different (half-closed) branches of $\{x\} \times \Sigma_k$ are
mapped by $F$ onto nondegenerate sets, i.e., onto (not necessarily
closed) arcs containing the point $F(x,q)$. Then there is a point in $\{f(x)\} \times \Sigma_2$ different from
$F(x,q)$ which is $F$-covered twice, by points $P,Q$ belonging to
different branches of $\{x\} \times \Sigma_k$. Hence, some open
arc $\{x\}\times J$ in the branch containing $P$ is such that the
closure of its $F$-image lies in the interior (in topology of
$M_{f(x)}$) of the $F$-image of the branch containing $Q$. Since
such a property carries over to all fibres close to the fibre over
$x$, the existence of a redundant open set for $F|_M$ easily
follows.

So, for every $x\in W_a$ there are at most two of $k$ branches of
$\{x\} \times \Sigma_k$ which are mapped by $F$ to nondegenerate
sets. If we denote by $J_1, \dots, J_k$ the branches of $\Sigma_k$
and by $W^i$ the set of all $x\in W_a$ with $F(\{x\}\times J_i) =
F(x,q)$, then  $W^i$ is closed in $W_a$ and, since $k\geq
3$, we have $W_a = \bigcup_{i=1}^k W^i$. Since $W_a\times \Sigma_k =
\bigcup_{i=1}^k (W^i \times \Sigma_k)$ and the sets $W^i \times
\Sigma_k$ are closed in $W_a\times \Sigma_k$, there is $i_0$
such that $W^{i_0} \times \Sigma_k$ has nonempty interior in
$W_a\times \Sigma_k$. It follows that $W^{i_0}$
has nonempty interior in $W_a$. Thus there is a set $\emptyset
\neq \Omega \subseteq W^{i_0}$ open in $W_a$.
So, if $A$ is an open arc lying in $J_{i_0}$, the set $\Omega \times A$
is open in $W_a \times \Sigma_k$, hence open in $M$. Since it is 
redundant  for $F|_M$, the proof is finished. \hfill $\square$
\end{proof}

\noindent {\bf Theorem C (full version).} {\it
Let $M$ be a minimal set (with full projection onto the base) of a fibre-preserving map in a compact graph bundle $(E,B,p,\Gamma)$. Assume that $M$ has nonempty interior.  Then the following holds.
\begin{enumerate}
\item [(C1)] $M = \Lint_s(M)$.
\item [(C2)] If $B$ is infinite then $M$ exhibits the following kind of `perfectness':
\begin{itemize}
\item[$\bullet$] If $U$ is a trivializing neighborhood, $h: p^{-1}(U) \to U\times \Gamma$ a canonical  homeomorphism and $\widetilde M_U = h(M_U)$, then for every $(z,p)\in \widetilde M_U$ there is a sequence of points $U\ni z_n \to z$, $z_n \neq z$, such that $(z_n,p) \in \widetilde M_U$ for all $n$.
\end{itemize}
\item [(C3)] $\mathcal R_B(M)$  is a closed nowhere dense subset of $B$.
\item [(C4)] All the sets $M_b$, $b\in B$, are unions of circles. In fact there exist an open dense set  $\mathcal{O} \subseteq B$ and a positive integer $m$ such that
\begin{itemize}
\item[$\bullet$] for each $z \in \mathcal{O}$, $M_z$ \emph{is} a \emph{disjoint} union of $m$ circles, and
\item[$\bullet$] for each $z\in B\setminus \mathcal{O}$, $M_z$ is a union of circles which properly contains a disjoint union of $m$ circles.
\end{itemize}
In particular, if $M_z$ is a circle for some $z\in B$, then $M_z$ is a circle for all $z$ in the open dense subset $\mathcal O$ of $B$.
\item [(C5)] For each $z\in \mathcal O$ there exists a trivializing neighborhood $z\in U \subseteq \mathcal O$ such that if $h: p^{-1}(U) \to U\times \Gamma$ is a canonical homeomorphism then $\widetilde M_U = h(M_U)$ has the structure of a direct product. It means that $\widetilde M_U = U \times \bigcup _{i=1}^m C_i$ where $C_1, \dots, C_m$ are pairwise disjoint circles in~$\Gamma$. Consequently,
\begin{itemize}
\item[$\bullet$] if $\mathcal{O} = B$, then $M$ is a sub-bundle of $E$ whose fibre is a disjoint union of $m$ circles, and
\item[$\bullet$] if $\mathcal{O} = B$, $E=B\times \Gamma$ and $B$ is connected, then $M$ is a direct product of $B$ and a disjoint union of $m$ circles.
\end{itemize}
\item [(C6)] The set $M_{\mathcal O}$ is dense in $M$.
\item [(C7)] Call a circle $\mathcal K \subseteq M_b$, $b\in B$, a \emph{generating circle} if there are circles $\mathcal K_n \subseteq M_{b_n}$, $b_n \in \mathcal O$, $n=1,2,\dots$, such that $\mathcal K_n \to \mathcal K$ with respect to the Hausdorff metric in $E$.
    Then the set $M$ is the union of all generating circles. If $b\in \mathcal O$ then $M_b$ is a disjoint union of $m$ circles and each of them is in fact generating. If $b\in B\setminus \mathcal O$, the set $M_b$ may contain a circle that is not generating but it always contains at least $m+1$ generating circles, at least $m$ of them being pairwise disjoint.
\item [(C8)] If $z\in \mathcal{O}$ then the set $M_z$, which is a disjoint union of $m$ circles, is mapped by $F$ onto a disjoint union of $m$ circles in $M_{f(z)}$.
\item [(C9)] If $z\in B\setminus \mathcal{O}$ then a generating circle in $M_z$ is mapped by $F$ onto a generating circle in $M_{f(z)}$. A non-generating circle in $M_z$ need not be mapped onto a circle.
\item [(C10)] If $f$ is monotone  then  $\mathcal{O} = B$ (hence, $M$ is a sub-bundle of $E$).
\item [(C11)] If $E = B\times \Gamma$ and $B$ is locally connected then $\mathcal{O} = B$ (hence, $M$ is a sub-bundle of~$E$ and if $B$ is also connected, then $M$ is a direct product).
\end{enumerate}}
\vspace{2mm}

Concerning (C8), let us remark that if $z\in \mathcal{O}$ and $S$ is a circle in $M_z$ then the map $F|_{S} :\, S \to M_{f(z)}$ need not be injective even if  $f$ is a homeomorphism (see the non-invertible skew-product torus map in \cite{KST}) and the map $F|_{M_{z}} :\, M_{z} \to M_{f(z)}$ need not be surjective (see Theorem D).

In (C9), two different/disjoint generating circles in $M_z$ can be mapped onto the same generating circle in $M_{f(z)}$ (again, see Theorem D).

\begin{proof}
\emph{(C1)} Since $\End (M) = \emptyset$, this follows from Lemma~\ref{L:Ends}.

\medskip

\emph{(C2)} Since the argument is local (concerns only that part of the minimal set which projects onto $U$), we may simply assume that $E=B\times \Gamma$, $M_U = U\times \Gamma$ and to work with $M_U$ rather than with $\widetilde M_U$. 

We have $(z,p) \in \Lint_s(M)$. Consider an $M$-open neighborhood $\mathcal G$ of $(z,p)$, mentioned in the definition of a strongly star-like interior point. One of the properties of $\mathcal G$ is that if $b\in p(\mathcal G)$ then $\mathcal G_b$ contains the point $(b,p)$. Thus, it is sufficient to prove that $z\in p(\mathcal G)$ is a limit point of $p(\mathcal G)$. Suppose, on the contrary, that $z$ is an isolated point of $p(\mathcal G)$. Then $\mathcal G \cap M_z$ is an $M$-open neighborhood of $(z,p)$. Since $F|_M$ is minimal, $(z,p)$ returns to $\mathcal G \cap M_z$ whence we obviously get that $z$ is a periodic point of $f$. However, $f$ is minimal and so $B$ is just the periodic orbit of $z$ under $f$, a contradiction with the infiniteness of $B$.

\medskip

\emph{(C3)} We claim that the set $\mathcal R_E(M)$ is closed (in $E$, hence also in $M$). To show this, let $\mathcal R_E(M) \ni \gamma_n \to \gamma \in E$. Since we work
only with a neighborhood of the fibre containing $\gamma$, we may
assume that $E=B\times \Gamma$. Denote $p(\gamma_n) = b_n$ and
$p(\gamma) = b$. Since $M$ is closed, $\gamma \in M$. However, $M=
\Lint_s(M)$ and so, by the definition of a star-like interior point,
for large $n$ the point $\gamma_n$ has an $M_{b_n}$-open
neighborhood whose second projection is a subset of the second
projection of an $M_b$-open neighborhood of the point $\gamma$.
Since $\gamma_n\in \mathcal R_E(M)$, this obviously implies that also $\gamma
\in \mathcal R_E(M)$. We have thus proved that $\mathcal R_E(M)$ is closed, hence
compact. Then also its projection $\mathcal R_B(M) = p(\mathcal R_E(M))$ is compact.

To prove that the (closed) set $\mathcal R_B(M)$ is nowhere dense, suppose, on the contrary, that some closed ball $C$ is a subset of
$\mathcal R_B(M)$ (closed balls here and in the rest of the proof of (C3) are always closed balls in the topology of $B$).

Combining Lemma~\ref{L:ramif1}(a) and Lemma~\ref{L:ramif2} we get that there are a closed ball $C_1 \subseteq C$ and ramification points $q_1, q_2 \in \Gamma$ such that $$C_1 \times \{q_1\} \subseteq \mathcal R_E(M) \quad \text{and} \quad F(C_1\times \{q_1\}) = f(C_1) \times \{q_2\} \subseteq \mathcal R_E(M)~.$$
The set $f(C_1) \subseteq \mathcal R_B(M)$ has nonempty interior in $B$ because $C_1$ has nonempty interior in $B$ and $f:B\to B$, being a minimal map, is feebly open. Then, by Lemma~\ref{L:ramif2}(b), there is a closed ball $C_2$ and a ramification point $q_3$ of $\Gamma$ such that
$$C_2 \subseteq f(C_1) \quad \text{and} \quad F(C_2 \times \{q_2\}) = f(C_2) \times \{q_3\} \subseteq \mathcal R_E(M)~.$$
Again, as above, $f(C_2)$ has nonempty interior in $B$ and so we can apply Lemma~\ref{L:ramif2} to find $C_3$ and $q_4$. Continuing in this way, we obtain a sequence of closed balls $(C_n)_{n=1}^\infty$ in $B$ and a sequence $(q_n)_{n=1}^\infty$ of ramification points of $\Gamma$ such that
$$
C_n \times \{q_n\} \subseteq \mathcal R_E(M)  \quad \text{and} \quad  F(C_n \times \{q_n\}) \supseteq C_{n+1}\times \{q_{n+1}\} \quad \text{for every $n$}~.
$$
Now choose a point $\gamma$ in the nonempty compact  set  
$$
(C_1 \times \{q_1\}) \cap F^{-1}(C_2\times \{q_2\}) \cap F^{-2}(C_3\times \{q_3\}) \cap \dots~.
$$
Then all the points $\gamma, F(\gamma), F^2(\gamma), \dots$ belong to $\mathcal R_E(M)$. By minimality of $F|_M: M\to M$, the set $\mathcal R_E(M)\subseteq M$ containing the $F$-orbit of $\gamma$ is dense in $M$. Since $\mathcal R_E(M)$ is also closed in $M$ (see the beginning of the proof of (C3)), we get that $\mathcal R_E(M)=M$. However, this contradicts the fact that $\mathcal R_E(M)$ is nowhere dense in $M$. Indeed, if $(z,g)\in \mathcal R_E(M)$ then it is a ramification point of $M_z$ and a small connected $\Gamma_z$-neighborhood of $(z,g)$ (which has the form of a star, a full sub-star of which is a subset of $M_z$) contains no other ramification points of $\Gamma_z$, while containing points from $M_z$ different from $(z,g)$.  It obviously follows that in every $M$-open neighborhood of $(z,g)$ there is an $M$-ball disjoint with $\mathcal R_E(M)$.

\medskip

\emph{(C4)} Let $H$ be the homeo-part of the minimal system $(B,f)$. Both the $f$-image and the $f$-pre-image of a nowhere dense set are nowhere dense (see Subsection~\ref{SS:minimality}). Therefore, since $\mathcal R_B(M)$ is nowhere dense in $B$ by (C3), the set
$$
H^*= H\setminus \bigcup _{n=-\infty}^\infty f^n(\mathcal R_B(M))
$$
is residual, $f(H^*)=H^*$, every point of  $H^*$ has just one $f$-pre-image, and both $f|_{H^*}$ and $(f|_{H^*})^{-1}$ are minimal homeomorphisms. For any $w\in H^*$, the set $M_w$ is a graph without end-points which, by definition of $H^*$, has no ramification point and so
$M_w$ is a circle or a disjoint union of several circles for all $w\in H^*$. 

Suppose that, for some $a\in B$, the set $M_a$ is not a union of circles. In our argument only $E_U$ for a small neighborhood $U$ of $a$ will play a role, therefore we may assume that $E=B\times \Gamma$. So, $M_a =\{a\}\times M_a^\Gamma$ for some subgraph $M_a^\Gamma$ of $\Gamma$. Choose $z_0 \in M_a^\Gamma$ such that $(a,z_0)\in M_a$ does not belong to any circle contained in $M_a$ (it may belong to a circle in $\Gamma_a$). Then for all $b\in B$ sufficiently close to $a$, in the set $M_b$ there is no circle containing the point $(b,z_0)$, since otherwise (due to closedness of $M$ and the fact that there are only finitely many circles in $\Gamma$) also $M_a$ would contain a circle containing $(a,z_0)$. Fix a point $y^* \in H^*$. Then its forward orbit under $f$ is a subset of $H^*$ and so, if we choose
a point $z\in \Gamma$ with $(y^*,z)\in M_{y^*}$, for each $n=0,1,2,\dots$ the point $F^n(y^*,z)$ belongs to one of the circles forming the set $M_{f^n(y^*)}$. It follows that the trajectory of $(y^*,z)$ under $F$ does not approach the point $(a,z_0)$, which contradicts the minimality of $F|_M$. Thus we have proved that
all the sets $M_b$, $b\in B$, are unions of circles.

Now let $m$ be the maximum number of (disjoint) circles in $M_w$ for $w\in H^*$. Then $m\geq 1$. Fix a point $w\in H^*$ such that $M_w$ consists of $m$ circles. Since $w$ has just one $f$-pre-image (and this pre-image belongs to $H^*$) and $F:M\to M$ is surjective, also $M_{f^{-1}(w)}$ consists of $m$ disjoint circles (less than $m$ circles cannot be continuously mapped onto $m$ disjoint circles). By induction, $M_{f^{-j}(w)}$ consists of $m$ disjoint circles for every $j=0,1,2,\dots$. Since $B$ is a compact metric space and $f:B\to B$ is minimal, the backward orbit $\{f^{-j}(w):\, j=0,1,2,\dots\}$ is dense in $B$ (see Subsection~\ref{SS:minimality}). Since $M$ is closed, the fact that $M_w$ consists of $m$ disjoint circles for every $w$ in a dense subset of $H^*$ implies (in view of the fact that there are only finitely many possibilities for a choice of $m$ disjoint circles in $\Gamma$) that
$M_w$ consists of $m$ disjoint circles for \emph{all} $w\in H^*$ and
$M_w$ contains $m$ disjoint circles (and perhaps some other circles) for all $w\in B\setminus H^*$.
So, if we put
$$
\mathcal O = \{z\in B:\, M_z \text{ is a disjoint union of $m$ circles}\},
$$
then
$
B\setminus \mathcal O = \{z\in B:\, M_z \text{ is a union of circles properly containing $m$ disjoint circles}\}.
$

Since $\mathcal O \supseteq H^*$, $\mathcal O$ is dense in $B$. To prove that $\mathcal O$ is open we are going to show that $B\setminus \mathcal O$ is closed. So, let $B\setminus \mathcal O \ni x_n \to x \in B$. Since we may assume that all the points $x_n$ are in a trivializing neighborhood of $x$, we may also assume that $E=B\times \Gamma$. Further, by passing to a subsequence if necessary, we may assume that for some disjoint circles $C_1,\dots, C_m$ in $\Gamma$ we have $M_{x_n} \supseteq \{x_n\}\times \bigcup_{i=1}^m C_i$ for every $n$. Taking into account that the points $x_n$ belong to $B\setminus \mathcal O$ and again passing to a subsequence if necessary, we may assume that there is a circle $S$ in $\Gamma$ different from all $C_i$, $i=1,\dots,m$, such that $M_{x_n}\supseteq \{x_n\}\times S$ for every $n$. Then, since $M$ is closed, $M_x \supseteq \{x\} \times (S \cup \bigcup_{i=1}^m C_i)$ which implies that $x\in B\setminus \mathcal O$.

\medskip

\emph{(C5)} Let $z\in V \subseteq \mathcal O$ be a trivializing neighborhood. We may simply assume that $p^{-1}(V) = V \times \Gamma$. Then $M_z = \{z\}\times \bigcup _{i=1}^m C_i$ for some pairwise disjoint circles $C_i$ in $\Gamma$. If a circle $C\subseteq \Gamma$ is different from these $m$ circles, then $M_v$ does not contain $\{v\}\times C$ whenever $v\in V$ is sufficiently close to $z$ (otherwise the closed set $M_z$ would contain $\{z\}\times C$). Thus, it is sufficient to choose a sufficiently small neighborhood $z\in U\subseteq V$.

From what we have just proved it follows that if $\mathcal O = B$ then $M$ is a bundle with fibre equal to a disjoint union of $m$ circles. Now additionally assume that $E=B\times \Gamma$ and $B$ is connected. For every $x\in B$ the set $M_x$ is a disjoint union of $m$ circles (where $m$ does not depend on $x\in B$). There are only finitely many $m$-tuples of circles in $\Gamma$ and so, using the closedness of $M$ and connectedness of $B$, we get that $M$ is the product of $B$ and some $m$-tuple of disjoint circles in $\Gamma$.

\medskip

\emph{(C6)} Since $f$ is minimal, the $f$-pre-image of a residual set is residual and so there is a point $x\in \mathcal O$ whose forward orbit is a subset of $\mathcal O$. Choose a point in $M_x$. Since its forward orbit is dense in $M$ and is a subset of $M_{\mathcal O}$, the result follows.

\medskip

\emph{(C7)} If $b\in \mathcal O$ then $M_b$ is a disjoint union of $m$ circles and each of them is generating by definition (even if the point $b$ is isolated in $B$). Then (C4), (C6) show that every $M_b$ is the union of generating circles (even if $b\in B\setminus \mathcal O$).

If $b\in B\setminus \mathcal O$, it is possible that each circle in $M_b$ is generating as in Theorem~D. However, it may contain also a non-generating circle. To see this, consider the case ($12_3$) in the proof of Theorem D. There, in one fibre of a minimal set, we can have two ``geometric" circles having two points in common. This gives $6$ circles altogether but only two of them, namely (in the notation from the proof of Theorem D) $\{c_l\}\times S_1$ and $\{c_l\}\times S_1^*$ , are generating ones. However, at least $m$ of the circles in $M_b$, $b\in B\setminus \mathcal O$ are disjoint generating circles. Indeed, consider a trivializing neighborhood $W$ of $b$ and think of $E_W$ as being the product $W\times \Gamma$. Then just choose a sequence of points $b_n\in \mathcal O$, $b_n\to b$ such that every $M_{b_n} = \{b_n\}\times A$ for the union $A$ of some fixed $m$ disjoint circles in $\Gamma$ (this is possible since $\Gamma$ contains only finitely many combinations of disjoint $m$ circles). Then $M_b \supseteq \{b\}\times A$ and so $M_b$ contains at least $m$ disjoint generating circles. Since $b\notin \mathcal O$, $M_b$ cannot be just the union of these $m$ circles and since we already know that $M$ is a union of generating circles, $M_b$ has to contain another generating circle.

\medskip

\emph{(C8)} Fix $z\in \mathcal{O}$.  First we  prove that if  $S$ is a circle in $M_z$ then $F(S)$ is a circle in $M_{f(z)}$.
We will work only with small neighborhoods of $z$ and $f(z)$, therefore we may assume that $E=B\times \Gamma$. By (C5), we may fix a neighborhood $z\in U \subseteq \mathcal O$ such that
\begin{equation}\label{E:strukturaMU}
M_U = U \times \bigcup _{i=1}^m C_i \text{ where $C_1, \dots, C_m$ are pairwise disjoint circles in $\Gamma$}.
\end{equation}
Set  $S=\{z\}\times C$  where $C$ is one of the circles $C_i$. We need to prove that $F(S) \subseteq M_{f(z)}$ is a circle.

Let us start by considering the case when $f(z) \in \mathcal O$. Then $M_{f(z)}$ is a disjoint union of circles and so $F(S)$ is necessarily a connected subset of one of them, call it $T$. To prove that $F(S)=T$ suppose, on the contrary, that $F(S)$ is a \emph{proper} subset of the circle $T$. We are going to prove that then there exists a redundant open set for $F|_M$ (which will contradict the minimality of $F|_M$). If $F(S)$ is an arc in $T$, there are two non-overlapping arcs in $S$ such that each of them is mapped onto $F(S)$. Hence there are also two disjoint arcs $\{z\}\times J_1$ and $\{z\}\times J_2$ in $S$ such that $F(\{z\}\times J_1)$ is in the interior of $F(\{z\}\times J_2)$. Then~(\ref{E:strukturaMU}) and the fact that the mentioned property of the point $z$ carries over to all the points sufficiently close to $z$, easily imply the existence of a redundant open set for $F|_M$, as desired. It remains to check the case when $F(S)$ is only a singleton in $T$. Then the existence of a redundant open set for $F|_M$ is obvious if also for all $v$ in a neighborhood of $z$ we have that $F(\{v\}\times C)$ is a singleton. If  such a neighborhood of $z$ does not exist, then arbitrarily close to $z$ there are points $v\in \mathcal O$ for which $F(\{v\}\times C)$ is not a singleton. By choosing such a point $v$ close enough to $z$ we can guarantee that $F(\{v\}\times C)$ is a \emph{proper} subset of a circle, i.e. an arc. To find a redundant open set for $F|_M$, one can simply repeat the argument which was used above in the case when $F(S)$ was an arc. We have thus proved that $F(S) \subseteq M_{f(z)}$
is a circle if $f(z) \in \mathcal O$. It is a generating circle by definition, since $f(z) \in \mathcal O$.

Now consider the case when $f(z)\in B\setminus \mathcal O$. In $U\setminus \{z\}$ there is a sequence $z_n \to z$ such that $f(z_n) \in \mathcal O$ (otherwise some neighborhood of $z$ would be mapped into $B\setminus \mathcal O$ which would contradict the fact that a minimal map sends open sets to sets with nonempty interior). Put $S_n = \{z_n\}\times C$ and $F(S_n) = \{f(z_n)\}\times K_n$, $n=1,2,\dots$. Then, by what we have proved above (note that both $z_n$ and $f(z_n)$ are in $\mathcal O$), we know that $K_n \subseteq \Gamma$ is a circle for every $n$. However, there are only finitely many circles in $\Gamma$ and so, by passing to a subsequence if necessary, we may assume that $K_n = K$ does not depend on $n$. Then obviously also $F(S) = \{f(z)\}\times K$ and so $F(S)$ is a circle, in fact a generating circle (because $f(z_n) \in \mathcal O$).

To finish the proof of (C8), it remains to show that different, hence disjoint, circles in $M_z$ are mapped onto \emph{disjoint} circles in $M_{f(z)}$.

Again, we start by considering a particular case when $f(z) \in \mathcal O$.
By replacing $U$ in~(\ref{E:strukturaMU}) by a smaller neighborhood of $z$ if necessary, we may assume, due to (C5), that $M_{f(U)} = f(U) \times \bigcup _{i=1}^m Q_i$ where $Q_1, \dots, Q_m$ are pairwise disjoint circles in~$\Gamma$. Let $S = \{z\}\times C$, $S' = \{z\}\times C'$ be disjoint circles in $M_z$ (here $C, C' \in \{C_1, \dots, C_m\}$, see~(\ref{E:strukturaMU})). To prove that also the circles $F(S)$ and $F(S')$ are disjoint, suppose on the contrary that $F(S)=F(S') = \{f(z)\}\times Q$ for some $Q\in \{Q_1,\dots,Q_m\}$. The circle $\{f(z)\}\times Q$ has positive distance from the rest of $M_{f(z)}$. Therefore, in view of~(\ref{E:strukturaMU}), for all $v$ sufficiently close to $z$ it holds that both $\{v\}\times C$ and $\{v\}\times C'$ are mapped by $F$ onto the same circle $\{f(v)\}\times Q$. The existence of a redundant open set for $F|_M$ easily follows; a contradiction.

Finally, consider the case when $f(z) \in B\setminus \mathcal O$. Again, let
$S = \{z\}\times C$, $S' = \{z\}\times C'$ be disjoint circles in $M_z$.
Choose a sequence $U\setminus \{z\} \ni z_n \to z$ such that $f(z_n) \in \mathcal O$. Consider the circles $S_n = \{z_n\}\times C$ and $S_n' = \{z_n\}\times C'$. For each $n$, both $z_n$ and $f(z_n)$ are in $\mathcal O$ and therefore, as we already know, $F(S_n) = \{f(z_n)\}\times P_n$ and $F(S_n') = \{f(z_n)\}\times P_n'$ are disjoint circles. By passing to a subsequence if necessary, we may assume that $P_n = P$ and $P_n' = P'$ do not depend on $n$. Then obviously $F(S) = \{f(z)\}\times P$ and $F(S') = \{f(z)\}\times P'$ which means that $F(S)$ and $F(S')$ are disjoint circles.

\medskip

\emph{(C9)} Let $S\subseteq M_z$ be a generating circle. So, there are circles $S_n \subseteq M_{z_n}$, $z_n \in \mathcal O$ (hence $z_n\neq z$), $n=1,2,\dots$, such that $S_n \to S$ with respect to the Hausdorff metric. By (C8), $F(S_n)$ is a generating circle for every $n$. Since $F(S_n) \to F(S)$ in the Hausdorff metric, $F(S)$ is a generating circle. Now see the proof of Theorem D, the case $(12_3)$. The set $M^*_{c_l}$ consists of two circles, one ``inside" the other. Together there are six circles there, two generating and four non-degenerating. Straightforward analysis shows that images of two non-degenerating circles are just arcs, not circles.

\medskip

\emph{(C10)} Let $f$ be monotone. Suppose that $B \setminus \mathcal O \neq \emptyset$. To show that this leads to a contradiction, consider two cases.

If for every $z\in B\setminus \mathcal O$ the set $f^{-1}(z)$ intersects $B\setminus \mathcal O$, then there is a backward orbit of $f$ lying in $B\setminus \mathcal O$. However, $B\setminus \mathcal O$ is nowhere dense while every backward orbit of a minimal map is dense, a contradiction.

If there exists $z_0 \in B\setminus \mathcal O$ such that the connected set $f^{-1}(z_0)$ is a subset of $\mathcal O$, we get a contradiction as follows. Fix a point $a\in f^{-1}(z_0)$. Since now we are going to find a special neighborhood of~$a$ by considering just small neighborhoods of $a$ and $z_0$, we may assume for a moment that $E=B\times \Gamma$. By (C5), there is a small neighborhood $U_a$ of $a$ such that $U_a \subseteq \mathcal O$ and $M_{U_a} = U_a \times \bigcup_{i=1}^m C_i^a$ where $C_1^a, \dots, C_m^a$ are pairwise disjoint circles in $\Gamma$. By (C8), $F(\{a\}\times C_i^a) = \{f(a)\}\times K_i^a$, $i=1,\dots, m$, for some pairwise \emph{disjoint} circles $K_1^a, \dots, K_m^a$ in $\Gamma$. Since there are only finitely many circles in $\Gamma$, there is $\varepsilon_0 >0$ such that any two different (not necessarily disjoint) circles in $\Gamma$ have Hausdorff distance at least $\varepsilon_0$. Therefore, if $i\in \{1,\dots,m\}$ and if $u\in U_a$ is sufficiently close to $a$ then the set $F(\{u\}\times C_i^a)$, which is a circle by (C8), equals $\{f(u)\} \times K_i^a$. By replacing $U_a$ by a smaller neighborhood if necessary, we may assume that the last claim works for all $u\in U_a$. Finally, consider the relative neighborhood of $a$ in $f^{-1}(z_0)$ of the form $V_a = U_a \cap f^{-1}(z_0)$. Denote also $S_i^a = \{z_0\} \times K_i^a$. Then we have that
\begin{equation}\label{E:Va}
\text{for every $v\in V_a$, $\quad F(M_v) = \bigcup_{i=1}^m S_i^a \subseteq M_{z_0}$}~.
\end{equation}
Without our above temporary assumption that $E=B\times \Gamma$, of course still a small relative neighborhood $V_a$ of $a$ exists such that~(\ref{E:Va}) works for some pairwise disjoint circles $S_1^a,\dots, S_m^a$ in $M_{z_0}$. Remember that, given $a\in f^{-1}(z_0)$, the family of these circles does not depend on the choice of $v\in V_a$.

Let $V_{a(1)}, \dots, V_{a(r)}$ be a finite cover of the compact space $f^{-1}(z_0)$ (in the relative topology), chosen from the open cover $\{V_a :\, a\in f^{-1}(z_0)\}$. Then, since $F(M)=M$, we have, by~(\ref{E:Va}),
\begin{equation}\label{E:Mz0}
M_{z_0} = \bigcup_{a\in f^{-1}(z_0)} F(M_a) =  \bigcup_{j=1}^r F(M_{V_{a(j)}}) = \bigcup_{j=1}^r \bigcup_{i=1}^m S_i^{a(j)}~.
\end{equation}
We claim that the family of $m$ disjoint circles $\{S_1^{a(j)}, \dots, S_m^{a(j)}\}$ does not depend on $j$. To see it, fix $j, k \in \{1,\dots, r\}$, $j\neq k$. In particular case when $V_{a(j)} \cap V_{a(k)} \neq \emptyset$ it suffices to choose $x\in V_{a(j)} \cap V_{a(k)}$ and to use  that, by~(\ref{E:Va}), it holds $\bigcup_{i=1}^m S_i^{a(j)} = F(M_x) = \bigcup_{i=1}^m S_i^{a(k)}$. In general case realize that in the family  $V_{a(1)}, \dots, V_{a(r)}$ there is a finite chain of sets starting with $V_{a(j)}$ and ending with $V_{a(k)}$ such that any two consecutive elements of the chain intersect (if such a chain did not exist, the \emph{connected} set $f^{-1}(z_0)$ would be a union of two \emph{disjoint} nonempty sets open in the topology of $f^{-1}(z_0)$). Hence also in the general case we have $\bigcup_{i=1}^m S_i^{a(j)} = \bigcup_{i=1}^m S_i^{a(k)}$. Then~(\ref{E:Mz0}) implies that $M_{z_0}$ is a union of just $m$ disjoint circles. Hence $z_0\in \mathcal O$, a~contradiction.

\medskip

\emph{(C11)} We claim that to prove  $\mathcal O =B$   we may without loss of generality assume that $B$ is also connected.
In fact, suppose for a moment that we have proved $\mathcal O =B$ under the additional assumption of connectedness of $B$. Then we can finish the proof as follows. The space $B$, being compact and locally connected, has finitely many components $B_1, \dots, B_r$ and these are locally connected. The map $f$, being minimal, cyclically permutes them and $f^r$ is minimal on each of them.  Then, for $i=1,\dots,r$, the set $M_{B_i}$ is a minimal set of $F^r|_{B_i \times \Gamma}$. Hence, using our temporary assumption that $\mathcal O$ is the whole base space provided the base space is locally connected and connected, we get that for every $x\in B_i$ the set $M_x$ is a disjoint union of $m_i$ circles (where $m_i$ does not depend on $x\in B_i$). There are only finitely many $m_i$-tuples of circles in $\Gamma$ and so, using the closedness of $M_{B_i}$ and connectedness of $B_i$, we get that $M_{B_i}$ is the product of $B_i$ and some $m_i$-tuple of disjoint circles in $\Gamma$. Further, by (C4) the positive integer $m_i$ does not depend on $i$, i.e. there is $m$ with $m_i=m$ for all $i=1,\dots,r$. Thus, $\mathcal O = B$ (still, if $r>1$, the $m_i$-tuple of circles may depend on $i$).

So, assume that the locally connected space $B$ is also \emph{connected}. We are going to prove that then $\mathcal O=B$.

Consider the open set $\mathcal O \subseteq B$ defined in (C4). Since $B$ is locally connected, so is $\mathcal O$. Recall that a space $X$ is locally connected if and only if for every open set $U$ of $X$, each  component of $U$ is open. It follows that the set $\mathcal O$ can be represented as a disjoint union $\mathcal O = \cup \,\mathcal W_j$ of a countable family of its components $\mathcal W_j$, each $\mathcal W_j$ being $B$-open, locally connected and connected. Let $m$ be the positive integer from (C4).

Due to the connectedness of $\mathcal W_j \subseteq \mathcal O$, by (C5) we obtain the direct product structure of each $M_{\mathcal W_j}$, i.e., there exist \emph{pairwise disjoint} circles $C_1^j,\dots,C_m^j$ in $\Gamma$ such that
\begin{equation}\label{E:MWj}
M_{\mathcal W_j} = \mathcal W_j \times \bigcup _{i=1}^m C_i^j~.
\end{equation}
The circles $C_1^j, \dots, C_m^j$ in general depend on $j$, but $m$ does not. Let $L$ be the (finite) set of all circles $C_i^j$ (for all $j$ and all $i=1,\dots,m$).

Since $M$ is a closed set, for the closure of $M_{\mathcal W_j}$ we have $\overline{M_{\mathcal W_j}} = \overline{\mathcal W_j} \times \bigcup _{i=1}^m C_i^j \subseteq M$.
The set $\overline{\mathcal W_j}$ is connected. We call each of $m$ connected components $\overline{\mathcal W_j} \times C_i^j$ of the closure $\overline{M_{\mathcal W_j}}$ a \emph{prime cylinder} (more precisely, $\overline{\mathcal W_j} \times C_i^j$ is a prime cylinder corresponding to the circle $C_i^j$). Each prime cylinder has nonempty  $E$-interior. Notice also that each prime cylinder is a union of generating circles and is of course a connected subset of $M$. For a fixed circle $C$ in $\Gamma$, consider the set of indices
$
I(C):= \{j:\, \mathcal W_j\times C \subseteq M\}.
$
Let $[C_\alpha]$ be the components of the set
$\overline{\cup_{j \in I(C)}\mathcal W_j} \times C$, so
\begin{equation}\label{E:cylmax}
\overline{\cup_{j \in I(C)}\mathcal W_j} \times C = \bigsqcup_{\alpha}\,[C_\alpha]~.
\end{equation}
We will say that each $[C_\alpha]$ is  a \emph{maximal cylinder} corresponding to the circle $C$ (note that it is a subset of $M$). Observe that $[C_\alpha]$ has the form
\begin{equation}\label{E:max-prim}
\hspace{0mm}[C_\alpha] =  \overline{\cup_{\gamma \in \Upsilon}P_\gamma} \ \ \text{where} \ P_\gamma, \gamma \in \Upsilon \text{ are some prime cylinders corresponding to $C$.}
\end{equation}
By definition, $[C_\alpha]\cap [C_\beta] = \emptyset$ for $\alpha \not=\beta$. We will also need the following claim.

\medskip

\noindent \emph{Claim (Properties of maximal cylinders).}
\begin{itemize}
\item [(a)] Two maximal cylinders $\mathcal M_1, \mathcal M_2$ corresponding to the same circle $C$ either are disjoint or coincide.
\item [(b)] If $b\in \mathcal O$ and $\mathcal M_1 \cap \mathcal M_2 \cap M_b\neq \emptyset$, then $\mathcal M_1 = \mathcal M_2$.
\item [(c)] If $\mathcal M$ and $\mathcal M_{\lambda}$, $\lambda \in \Lambda$ are maximal cylinders with $M\subseteq \cup_{\lambda \in \Lambda} M_{\lambda}$, then $\mathcal M = \mathcal M_{\lambda_0}$ for some $\lambda_0 \in \Lambda$.
\item [(d)] For any $k\geq 1$, the $F^k$-image of a prime cylinder $P=\overline{\mathcal W_i}\times C \subseteq M$ is a subset of a maximal cylinder.
\item [(e)] The family of all maximal cylinders is finite and its union equals $M$.
\item [(f)] Any maximal cylinder is mapped by $F$ into a maximal cylinder.
\end{itemize}

\smallskip

\noindent \emph{Proof of Claim} (a) Each of the sets $\mathcal M_1, \mathcal M_2$ is a component of the set $\overline{\cup_{j \in I(C)}\mathcal W_j} \times C$. Two non-disjoint components coincide.

(b) By (C4), $M_b$ is a \emph{disjoint} union of circles. One of them, call it $C$, is such that $\mathcal M_1 \cap \mathcal M_2$ intersects $\{b\}\times C$ and, by definition of maximal cylinders, $(\mathcal M_1)_b = (\mathcal M_2)_b = \{b\}\times C$. Now apply (a).

(c) By definition of a maximal cylinder, there exists $b\in \mathcal O$ and $C\in L$ such that $\mathcal M \supseteq \{b\}\times C$. There is $\lambda_0 \in \Lambda$ such that $\mathcal M_{\lambda_0}$ intersects $\{b\}\times C$. By (b), $\mathcal M = \mathcal M_{\lambda_0}$.

(d) The set $P$ is a union of generating circles and, by (C8) and (C9), a generating circle is mapped onto a (generating) circle.  It follows that if $S$ is a circle in $\Gamma$ then, due to continuity of $F^k$ and the fact that $\Gamma$ contains only finitely many circles, the set of those $z\in \overline{\mathcal W_i}$ for which $F^k(\{z\}\times C) = \{f^k(z)\}\times S$, is open in $\overline{\mathcal W_i}$. However, the set $\overline{\mathcal W_i}$ is connected. Therefore there exists one circle $S$ such that
\begin{equation}\label{E:cyl1}
F^k(P) = f^k(\overline{\mathcal W_i})\times S \subseteq M.
\end{equation}
Since $f: B\to B$ is minimal, it is feebly open. Hence $f^k$ is feebly open. Therefore the set $U_i:= \Inte f^k(\mathcal W_i)$ of all $B$-interior points of  $f^k(\mathcal W_i)$ is dense in $f^k(\mathcal W_i)$, hence also dense in $\overline{f^k(\mathcal W_i)} = f^k(\overline{\mathcal W_i})$. So, $\overline U_i = f^k(\overline{\mathcal W_i})$. On the other hand, $U_i$ is open and $\mathcal O = \cup_j \,\mathcal W_j$ is dense and open, hence
\begin{equation}\label{E:cyl2}
f^k(\overline{\mathcal W_i}) = \overline U_i =  \overline{\mathcal O \cap U_i} =  \overline {(\cup_j\,\mathcal W_j) \cap U_i} = \overline {\cup_j\,(\mathcal W_j \cap U_i)} ~.
\end{equation}
Further, by~(\ref{E:cyl1}), $U_i\times S \subseteq M$. So, if $\mathcal W_l \cap U_i \neq \emptyset$ for some $l$, then $\mathcal W_l \times S \subseteq M$, i.e. $l\in I(S)$. It follows, taking into account~(\ref{E:cyl1}) and~(\ref{E:cyl2}), that
$$
F^k(P) = f^k(\overline{\mathcal W_i})\times S= \overline {\cup_{l\in I(S)}\, (\mathcal W_l \cap U_i)} \times S \subseteq \overline{\cup_{l \in I(S)}\mathcal W_l} \times S = \bigsqcup_{\alpha}\,[S_\alpha]
$$
where $[S_\alpha]$ are the components of the set $\overline{\cup_{l \in I(S)}\mathcal W_l} \times S$ (see (\ref{E:cylmax})). Since $F^k(P)$ is connected, it is a subset of one $S_{\alpha}$, which finishes the proof that $F^k$-image of a prime cylinder is a subset of a maximal cylinder.

(e) Now, since the prime cylinder $P$ has nonempty interior in $M$ and $F|_M$ is minimal, we have that $M = \bigcup_{k=0}^{N-1} F^k(P)$ for some $N$ (this is a property of minimal systems, see Subsection~\ref{SS:minimality}). This together with (d) give that $M$ is covered by $N$ (not necessarily distinct) maximal cylinders. Then, using (c), we get that the family of all maximal cylinders is finite (has at most $N$ elements) and its union equals $M$.

(f) Let $\mathcal M_1, \dots, \mathcal M_r$ be the list of all (pairwise distinct) maximal cylinders (at the moment we do not know whether they are pairwise disjoint). For $i=1,\dots, r$ put $\mathcal M_i = B_i \times S_i$, where $B_i \subseteq B$ is closed and connected set (containing at least one of the sets $\mathcal W_j$) and $S_i$ is a circle in $\Gamma$ (in fact $S_i\in L$, see~(\ref{E:MWj}) and the notation $L$ after it).
We prove that, for instance, $F(\mathcal M_1)$ is a subset of a maximal cylinder. By~(\ref{E:max-prim}),
$$
\mathcal M_1= \overline{\cup_{\gamma \in \Upsilon}P_\gamma} \quad \text{where } P_\gamma, \gamma \in \Upsilon \text{ is the family of prime cylinders contained in $\mathcal M_1$}
$$
(of course, all these prime cylinders $P_\gamma$ correspond to the circle $S_1$). We know, by (d), that for each $\gamma \in \Upsilon$ there is a maximal cylinder, call it $\mathcal N_{\gamma}$, with $F(P_\gamma) \subseteq \mathcal N_{\gamma}$. In the particular case when all these maximal cylinders are the same, i.e. when there is $\gamma_0 \in \Upsilon$ such that $\mathcal N_{\gamma} = \mathcal N_{\gamma_0}$ for all $\gamma \in \Upsilon$, we get the desired relation:
$$
F(\mathcal M_1)= F(\overline{\cup_{\gamma \in \Upsilon}P_\gamma}) = \overline{F(\cup_{\gamma \in \Upsilon}P_\gamma)} = \overline{\cup_{\gamma \in \Upsilon}F(P_\gamma)} \subseteq \mathcal N_{\gamma_0}~.
$$
To finish the proof, we are going to show that the assumption that not all maximal cylinders $\mathcal N_{\gamma}$ are the same, leads to a contradiction.

So, let $d\geq 2$ and $\mathcal N^{1}, \dots, \mathcal N^{d}$ be the list of all pairwise distinct maximal cylinders in the family $\mathcal N_{\gamma}$, $\gamma \in \Upsilon$. Then there is a decomposition $\Upsilon = \Upsilon^1 \sqcup \dots \sqcup \Upsilon^d$ such that $F(P_\gamma) \subseteq \mathcal N^j$ for all $\gamma \in \Upsilon^j$. Denote $\Pi_j:= \overline{\cup_{\gamma \in \Upsilon^j}P_\gamma}$, $j =1,2, \dots, d$. Of course,
$$
F(\Pi_j) = F(\overline{\cup_{\gamma \in \Upsilon^j}P_\gamma})\subseteq \mathcal N^j~.
$$
We claim that the sets $\Pi_j$ are pairwise disjoint. To show this, suppose on the contrary that $\Pi_i \cap \Pi_k \neq \emptyset$ for some $i\neq k$. Then, in view of the fact that all prime cylinders $P_\gamma$ correspond to the circle $S_1$, there is $b_0\in B$ such that $\{b_0\} \times S_1 \subseteq \Pi_i\cap \Pi_k$. Obviously, $\{b_0\} \times S_1$ is a generating circle and, by (C8) and (C9), its $F$-image is some circle $\{f(b_0)\} \times S^*$. Then $\mathcal N^i = B^i \times S^*$ and $\mathcal N^k = B^k \times S^*$ for some closed sets $B^i, B^k$ containing $f(b_0)$. Here $\mathcal N^i$, $\mathcal N^k$ are different, but not disjoint, maximal cylinders corresponding to the same circle $S^*$. This contradicts already proved part (a) of the Claim. So, we have proved that the sets $\Pi_j$ are pairwise disjoint. Then
$$
\mathcal M_1= \overline{\cup_{\gamma \in \Upsilon}P_\gamma} = \cup_{j=1}^d\overline{\cup_{\gamma \in \Upsilon^j}P_\gamma} = \sqcup_{j=1}^d \Pi_j
$$
is the decomposition of the connected set $\mathcal M_1$ into finitely many closed nonempty sets, a contradiction. This finishes the proof of Claim.  \hfill $\checkmark$$\checkmark$$\checkmark$

\medskip

Now we are ready to finish the proof of (C11).

\medskip

Similarly as in the proof of the part (f) of Claim, let $\mathcal M_1, \dots, \mathcal M_r$ be the list of all (pairwise distinct) maximal cylinders, where $\mathcal M_i = B_i \times S_i$. By the part (b) of Claim, two different maximal cylinders may intersect only in fibres above the set $B\setminus \mathcal O$. Therefore
\begin{equation}\label{E:podstatne}
\mathcal M_1\setminus (\mathcal M_2 \cup \dots \cup \mathcal M_r) \quad \text{ has nonempty interior in $M$}~.
\end{equation}
Since the map $F|_M$ is minimal, there exists a positive integer $j$ with
$$
F^j(\mathcal M_1) \cap \left(\mathcal M_1\setminus (\mathcal M_2 \cup \dots \cup \mathcal M_r)\right) \neq \emptyset~.
$$
However, every maximal cylinder is mapped by $F$ into a maximal cylinder, therefore we necessarily have $F^j(\mathcal M_1) \subseteq \mathcal M_1$. It follows that
$F^j|_{\mathcal M_1} : \mathcal M_1 \to \mathcal M_1$ is minimal. (Indeed, if in a minimal system
$(M,F)$ there is a closed and connected set $Y\neq \emptyset$ with $F^j (Y) \subseteq Y$ for some $j\geq 1$, then $Y$ is a minimal set of $F^j$. This is probably well known and explicitly can be found in~\cite{LS}.)

However, $f$ is minimal on the \emph{connected} space $B$ (see the discussion at the beginning of the proof of (C11)), hence it is totally minimal (this is well known, see e.g.~\cite{LS}). Since the minimal map $f^j$ is the base map of $F^j$, the fact that $F^j|_{\mathcal M_1}$ is minimal implies that $B_1 = B$. In the same way we get $B_i = B$ for all $i=1,\dots, r$. But then $M$ is a finite union (see Claim (e)) of maximal cylinders in the form $\mathcal M_i = B \times  S_i$, $i=1,\dots, r$. Since the maximal cylinders $\mathcal M_i$ of this particular form are assumed to be pairwise distinct and $\mathcal O \neq \emptyset$, by Claim (b) they are pairwise disjoint (i.e, the circles $S_i$ are pairwise disjoint). Thus $\mathcal O = B$. The sets $\mathcal M_i$ are the components of the minimal set $M$ and so they are cyclically permuted by $F$.
\hfill $\square$
\end{proof}


\section{Proof of Theorem~E}\label{S:proofsDE}


\begin{lemma}\label{L:totally disconnected}
Let $M$ be a nowhere dense closed subset of a compact graph bundle $(E,B,p,\Gamma)$. Then a typical fibre of $M$ is totally disconnected.
\end{lemma}

\begin{proof}
Suppose, on the contrary, that $A=\{b\in B:\, \text{$M_b$ is not totally disconnected}\}$ is of 2nd category in $B$. Of course, $M_b$ is not totally disconnected if and only if it contains an arc and since $\Gamma$ is a graph, this is if and only if $M_b$ contains a ball in $\Gamma_b$. Therefore, since $A$ is of 2nd category, there is $n_0 \in \mathbb N$ such that
$$
A_{n_0} = \{b\in B:\, \text{$M_b$ contains a ball in $\Gamma_b$ with radius $\geq 1/n_0$}\}
$$
is of 2nd category. Since $B$ is covered by finitely many trivializing neighborhoods, there is a trivializing neighborhood $U$ such that $A_{n_0} \cap U$ is of 2nd category. To get a desired contradiction, it is sufficient to show that $M\cap p^{-1}(U)$ is somewhere dense. Of course, we may without loss of generality assume that $p^{-1}(U) = U\times \Gamma$. To prove that $M\cap (U\times \Gamma)$ is somewhere dense, fix a countable dense set $S\subseteq \Gamma$. For $b\in U$ and $s\in S$, a ball in $\Gamma_b = \{b\}\times \Gamma$ whose radius is $\geq 1/n_0$ and whose center has distance from $\{b\}\times \{s\}$ less than $1/(2n_0)$ is in the sequel said to be \emph{a big ball centered close to level $s$}. Let
$$
A^{U,s}_{n_0} = \{b\in U:\, \text{$M_b$ contains a big ball centered close to level $s$}\}~.
$$
It is obvious that $A_{n_0}\cap U = \bigcup_{s\in S} A^{U,s}_{n_0}$ and so there is $s_0\in S$ such that $A^{U,s_0}_{n_0}$ is of 2nd category, hence dense in some nonempty open set $G_1 \subseteq U$. On the other hand, any ball in $\Gamma$ whose radius is $\geq 1/n_0$ and whose center has distance from $s_0$ less than $1/(2n_0)$, contains the ball $G_2$ with center $s_0$ and radius $1/(2n_0)$. Then  $M$,  being closed, contains the open set $G_1\times G_2$. This contradicts  nowhere density of $M$.  \hfill $\square$
\end{proof}

\begin{proposition}\label{P:D}
In Theorem~A, suppose that $\card M_{z} <~\infty$ for some $z$ in the homeo-part $H$ of $f$. Then
a typical fibre of the minimal set $M$ has cardinality $N:= \min \{\card M_x:\, x\in H\} < \infty$.
\end{proposition}

\begin{proof}
By the assumption, $N\leq \card M_{z}$ is a positive integer and there is $z_0\in H$ with $\card M_{z_0} = N$. Denote
$$
B^{(\leq N)}:=\{x\in B:\, \card M_x \leq N\}~.
$$
Then $z_0 \in B^{(\leq N)}$ and we claim that $B^{(\leq N)} = \bigcap _{n=1}^\infty G^{(N)}_{1/n}$ where $G^{(N)}_{1/n}$ is the set of those points $x\in B$ for which $M_x$ can be covered by a disjoint union of $N$ open sets in the fibre $p^{-1}(x)$, each of these sets having diameter $<1/n$. The inclusion $B^{(\leq N)} \subseteq \bigcap _{n=1}^\infty G^{(N)}_{1/n}$ is trivial. To prove the converse inclusion, realize that simultaneous assumptions $x\in \bigcap _{n=1}^\infty G^{(N)}_{1/n}$ and $\card M_{x} \geq N+1$ obviously give a contradiction. Since $M$ is compact, $G^{(N)}_{1/n}$ is open. So,  $B^{(\leq N)}$ is a $G_{\delta}$ set in~$B$. Moreover, we claim that it is dense in $B$. To see this, realize that the set $H$ is $f$-invariant and for $x\in H$ we have $\card M_x \geq \card M_{f(x)} \geq N$. Hence, since for $z_0\in H$ we have $\card M_{z_0} = N$, we get that $\card M_{f^k(z_0)} = N$ for all $k=0,1,\dots$. So the set $B^{(\leq N)}$ contains the whole (forward) orbit of $z_0$, which is dense by minimality of $f$. We have proved that $B^{(\leq N)}$ is a $G_{\delta}$ dense set in $B$.

For each $x$ in the $G_{\delta}$ dense set $H \cap B^{(\leq N)}$ it obviously holds $\card M_x = N$. \hfill $\square$
\end{proof}

\noindent {\bf Theorem E.} {\it
Let $M$ be a minimal set (with full projection onto the base) of a fibre-preserving map in a compact graph bundle $(E,B,p,\Gamma)$. Assume that $M$ is nowhere dense.  Then either
\begin{enumerate}
\item [(E1)] a typical fibre of $M$ is a Cantor set, or
\item [(E2)] there is a positive integer $N$ such that a typical fibre of $M$ has cardinality $N$.
\end{enumerate}}

\vspace{2mm}

\begin{proof} 

Below, we use some ideas from the proof of~\cite[Theorem 6.1]{HY}.

By Lemma~\ref{L:totally disconnected}, a typical fibre of $M$ is totally disconnected. If (E1) does not hold, 
\begin{equation}
B^{\text {isol}}:= \{b\in B:\, M_b \text{ has an isolated point}\} \notag
\end{equation}
is a  2nd category set. Since $B$ can be covered by a finite number of trivializing neighborhoods, one of them, call it $W$, is such that $B^{\text {isol}}\cap W$ is of 2nd category in $W$ (even, of 2nd category in $B$).

We fix a homeomorphism $h: p^{-1}(W)\to W\times \Gamma$ such that on $p^{-1}(W)$ it holds $\pr_1 \circ h = p$.

Let $\mathcal T$ be a countable family of subtrees of the fibre $\Gamma$ such that the interiors, in the topology of $\Gamma$, of them are connected (i.e., the interior of such a tree is obtained from the tree by possible removing of some or all of the endpoints of the tree; no point which is not an endpoint is removed) and these interiors form a base of the topology on $\Gamma$. Consider the countable set
\begin{equation}
\mathcal D:= \left \{ \left ( T_1^{\Gamma}, T_2^{\Gamma}  \right ) :\,  T_1^{\Gamma}, T_2^{\Gamma} \in \mathcal T \, \text{ and } \,  T_1^{\Gamma} \subseteq \Inte T_2^{\Gamma}  \right \}.   \notag
\end{equation}
Note that the homeomorphism $h$ induces corresponding families of trees in each fibre $p^{-1}(b)$, $b\in W$. For each pair
$(T_1^{\Gamma}, T_2^{\Gamma}) \in \mathcal D$, put
\begin{equation}\label{E:defWT}
W( T_1^{\Gamma}, T_2^{\Gamma}):= \{b\in W:\, M_b \cap T_{2,b} = M_b \cap T_{1,b} \, \text { is a singleton}\}
\end{equation}
where $T_{i,b}:= h^{-1}(\{b\} \times T_i^{\Gamma})$ is the tree in $p^{-1}(b)$ corresponding to $T_i^{\Gamma}$, $i=1,2$. Of course,
\begin{equation}
B^{\text {isol}} \cap W = \bigcup _{(T_1^{\Gamma}, T_2^{\Gamma}) \in \mathcal D} W( T_1^{\Gamma}, T_2^{\Gamma})~. \notag
\end{equation}
Since $B^{\text {isol}} \cap W$ is a 2nd category set, there is a pair $(\widetilde T_1^{\Gamma}, \widetilde T_2^{\Gamma}) \in \mathcal D$ such that $W( \widetilde T_1^{\Gamma}, \widetilde T_2^{\Gamma})$ is dense in an open set $U\subseteq W$.

Let $\mathcal K(E)$ be the (compact) space of all compact subsets of $E$ endowed with the Hausdorff distance generated by the original metric in $E$. Since $M$ is compact, the map $\Theta : B\to \mathcal K(E)$ defined by $\Theta (b) = M_b$, $b\in B$ is upper semicontinuous. Hence, see e.g.~\cite[Theorem 1.4.13]{AF}, the set $C(\Theta)$ of continuity points of $\Theta$ is residual in $B$. By Lemma~\ref{L:homeo-part in P}, there is an invariant  residual set $R$ in $B$ such that $R\subseteq C(\Theta)\cap H$ where $H$ is the homeo-part of $f$.

Denote $V:= \Inte \widetilde T_2^{\Gamma}$. We claim that for any $b\in U\cap R \subseteq  \overline{W( \widetilde T_1^{\Gamma}, \widetilde T_2^{\Gamma})} \cap R$ it holds that $M_b^{\Gamma}\cap V$ is a singleton. In fact, each such point $b$ is a limit of points from $W( \widetilde T_1^{\Gamma}, \widetilde T_2^{\Gamma})$ and so $M_b^{\Gamma}\cap \widetilde T_1^{\Gamma}$ contains a point. Suppose that $M_b^{\Gamma}\cap V$ contains more than one point. Then, since $b$ is a point of continuity of $\Theta$, also for those points $c \in W( \widetilde T_1^{\Gamma}, \widetilde T_2^{\Gamma}) \cap U$ which are sufficiently close to $b$, we get that $M_c^{\Gamma} \cap V$ contains at least two points, which contradicts~(\ref{E:defWT}).

The set $\mathcal O := M \cap h^{-1}(U\times V)$ is a nonempty open subset of $M$. Hence, by the well known property of compact minimal systems, there is a positive integer $n_0$ such that every point from $M$ visits $\mathcal O$ not later than after $n_0$ iterations.

Now fix $y\in R$ and $e\in M_y$. By what was said above, $F^{n(e)}(e) \in \mathcal O$ for some $n(e)\leq n_0$. Hence $F^{n(e)}(G(e)) \subseteq \mathcal O$ for some neighborhood $G(e)$ of $e$ in  $M_y$. It follows that
\begin{equation}\label{E:vojde do O}
F^{n(e)}(G(e)) \subseteq h^{-1}(\{ f^{n(e)}(y) \} \times V) \cap M \subseteq \mathcal O~.
\end{equation}
By definition of $\mathcal O$ and the fact that $y$ has been chosen in the invariant set $R$, we get that $f^{n(e)} (y) \in U\cap R$. Therefore, by~(\ref{E:vojde do O}), $F^{n(e)}(G(e))$ is a singleton. Then  also $F^{n_0}(G(e))$ is a singleton. Since $M_y$ is compact, there are finitely many points $e_1,\dots, e_k \in M_y$ such that $G(e_1) \cup \dots \cup G(e_k) = M_y$. It follows that $F^{n_0}(M_y)$ is a finite set (it is a subset of $M_{f^{n_0}(y)}$ with cardinality $\leq k$).
Since $y\in H$ and $F(M)=M$, also $f^{n_0}(y)\in H$ and $M_{f^{n_0}(y)}= F^{n_0}(M_y)$ is finite. So, one can apply Proposition~D to get (E2).
\end{proof}

\begin{acknowledgements} The first and second authors were supported by Max-Planck-Institut f\"ur Mathematik (Bonn); they acknowledge the hospitality of the Institute. 
The second author was also supported by the Slovak Grant Agency, grant number VEGA~1/0978/11 and by the Slovak Research and Development Agency under the contract No.~APVV-0134-10. The second and third authors were partially supported by Fondo Nacional de Desarrollo Cient\'ifico y Tecnol\'ogico (Chile), project 1110309 and by Comisi\'on Nacional de Investigaci\'on Cient\'ifica y Tecnol\'ogica  (Chile), program ACT-56.
\end{acknowledgements}

\end{document}